\documentclass[twoside,11pt]{article}

\RequirePackage{xr}
\RequirePackage[OT1]{fontenc}
\RequirePackage{amsthm,amsmath,amssymb,dsfont,bm}
\RequirePackage{natbib}
\RequirePackage{algorithm}
\RequirePackage{algpseudocode}
\RequirePackage{graphicx} 
\RequirePackage[section]{placeins}
\RequirePackage{color}
\RequirePackage{chngcntr}
\RequirePackage{enumerate}
\usepackage[shortlabels]{enumitem}
\RequirePackage{longtable}
\usepackage{multirow}
\usepackage{mathtools}
\usepackage[english]{babel} 
\usepackage{xcolor}
\usepackage{dutchcal}
\usepackage{mathrsfs}
\usepackage{tikz}
\usetikzlibrary{arrows.meta,positioning,calc}

\RequirePackage[OT1]{fontenc} \RequirePackage{amsthm}
\RequirePackage{amsfonts} \RequirePackage{amssymb}
\RequirePackage{bbm}

\usepackage[normalem]{ulem}

\RequirePackage{graphicx}
\RequirePackage{verbatim}
\usepackage{url}

\numberwithin{equation}{section}

\let\manuscriptproof\proof
\let\endmanuscriptproof\endproof

\usepackage[preprint]{jmlr2e}

\let\proof\manuscriptproof
\let\endproof\endmanuscriptproof
\DeclareUnicodeCharacter{211D}{\ensuremath{\mathbb{R}}}
\makeatletter

\expandafter\let\csname c@definition\endcsname\relax

\expandafter\let\csname c@theorem\endcsname\relax

\expandafter\let\csname c@corollary\endcsname\relax

\expandafter\let\csname c@example\endcsname\relax

\expandafter\let\csname c@remark\endcsname\relax

\expandafter\let\csname c@lemma\endcsname\relax
\makeatother

\theoremstyle{plain}
\newtheorem{definition}{Definition}[section]
\newtheorem{theorem}{Theorem}[section]
\newtheorem{corollary}{Corollary}[section]
\newtheorem{example}{Example}[section]
\newtheorem{remark}{Remark}[section]
\newtheorem{lemma}{Lemma}[section]

\def\@bysame#1{\vrule height 1.5pt depth -1pt width 3em \hskip
0.5em\relax}

\newcommand{\N}{ \mathbb{N} }

\newcommand{\R}{ \mathbb{R} }

\newcommand{\calE}{\mathcal{E}}
\newcommand{\calF}{\mathcal{F}}
\newcommand{\calG}{\mathcal{G}}

\newcommand{\calM}{\mathcal{M}}
\newcommand{\calN}{\mathcal{N}}
\newcommand{\calS}{\mathcal{S}}

\newcommand{\calX}{\mathcal{X}}

\newcommand{\calZ}{\mathcal{Z}}

\newcommand{\eins}{{\bm 1}}

\newcommand{\matS}{{\bm S}}

\newcommand{\vecnull}{{\bm 0}}

\newcommand{\vecx}{{\bm x}}

\newcommand{\Var}{{\mbox{Var\,}}}
\newcommand{\Cov}{{\mbox{Cov\,}}}

\newcommand{\one}{ \mathbb{I} }

\usepackage{subfigure}
\usepackage{mathtools} 
\usepackage{bm}
\usepackage{dsfont}
\usepackage{color}
\usepackage{amsthm}

\usepackage{float}                 

\begin{document}

\title{Context-Adaptive Thresholding for Conditionally Representative Monitoring and Classification}
\author{\name Ansgar Steland \email steland@stochastik.rwth-aachen.de\\
\addr Institute of Statistics and AI Center,\\
RWTH Aachen University,\\
Aachen, Germany, \\ Date: \today}
\ShortHeadings{Context-Adaptive Thresholding}{Steland}
\firstpageno{1}
\maketitle

\begin{abstract}
Commonly, classifiers and monitoring procedures  are trained from  labeled data by optimizing an objective such as the misclassification rate. This may lead to unrepresentative conditional distributions of the outcome (the labels) given important external variables, different from the conditional laws in the population. We show how to modify any given threshold-type classifier resp. monitoring rule to achieve representative conditional label prediction by using adapting the threshold to a covariate $Z$ (the context) to distribute sensitivity while maintaining the false alarm rate. In case that the alarm event is unknown, this approach also allows to (approximately) infer the event in terms of a thresholding rule. The approach is implemented by a computationally cheap nonparametric estimation procedure, and its properties are studied in terms of nonasymptotic error bounds and  asymptotic distribution theory including empirical process theory. These results allow to construct uniform confidence bands, functional hypothesis tests and change-detection procedures. For the well known FICOS credit scoring example, often used in interpretable machine learning, threshold adaptation leads to an easily interpretable decision rule which can compete with state of the art methods including transformers, in terms of common classification metrics. 
\end{abstract}


\section{Introduction}
	
	The determination of an optimal threshold for a continuously observed feature is a pervasive challenge in diverse fields of inference. In areas such as  monitoring of processes or anomaly detection, such a feature $X$ is observed and a binary signal is generated by comparing $X$ against a fixed threshold $c$. The event of interest, often interpreted as indicative for an abnormal condition, is then defined as $\{X  > c\} $. A primary objective in this context is to select $c$ such that the alarm probability, $P( X > c )$, remains below a pre-specified level under baseline (normal) operating conditions thus ensuring a controlled false-alarm rate. This typically involves estimating $c$ from an initial learning sample, $X_1, \ldots, X_n$, drawn from the nominal distribution, before applying the derived rule to a sequential data stream $X_{n+1}, X_{n+2}, \ldots$. The subsequent monitoring usually aims to detect deviations from the baseline distribution, which increase the alarm rate. This framework forms the basis for various statistical process control methods and change-point detection algorithms (e.g., \cite{Shiryaev1963}; \cite{Siegmund1985}; \cite{Basseville1993}; \cite{steland2026online}).
	
	Beyond direct monitoring, thresholding plays a critical role in binary classification and risk prediction. In this domain, the true event of interest, $E$, and perhaps even its indicator  $Y = \mathbf{1}_E$ might not be directly observable, but a predictor variable $X$  correlated with $Y$ is available. For instance, in credit risk assessment, $Y=1$ might signify a loan default, while $X$ represents a credit score. More generally, $ X = g(\xi) $ may be a given classifier based on input features $ \xi $ which we aim to optimize a posteriori. A common predictive strategy is to classify the outcome $ X $ by the rule $\hat{Y} = \mathbf{1}_{\{X > c\}}$, where $c$ is a threshold selected ad hoc, often as $ 1/2$, or learned from the training sample $(Y_1, X_1), \ldots, (Y_n, X_n)$. The selection of $c$ is often driven by criteria such as the minimization of the misclassification error $P(Y \neq \hat{Y})$, \cite{mohammadi2003threshold}, or optimizing other performance metrics like precision, recall, or F1-score, see  \cite{Hastie2009}; \cite{Ripley1996},  \cite{LiptonElkanNara2014}, \cite{NIPS2014_98053046} and \cite{ZengJiangChengDobriban26-fair-classif}, amongst others. Fundamentally, this problem can be viewed as inferring an unobserved event $ E $ by judiciously thresholding an observed, correlated variable $X$.
	
	A more intricate scenario we are interested in arises, when context is given via a  covariate $Z$ whose influence on the event of interest $Y$ is non-negligible, and, consequently, the conditional probability $ \pi(z) = P(Y=1|Z=z) $ is of interest. Here, it is not assumed that $Z$ is a protected variable, but somehow correlated to $Y$. For example, in the context of credit loans, $Z$ could represent the loan amount. It may be desired that the classifier $\hat{Y}$ is {\em fair} and {\em aligned to the ground truth} in the sense that the conditional law of $ \hat{Y} $ given $ Z $ coincides with the ground-truth law $ \pi(z) $. We shall call this notion of fairness {\em representative fairness}. From a forecast perspective this means that the forecast $ \hat{Y}$ is {\em probabilistically calibrated given $Z$}. It ensures that the label distributions across the values of $Z$ do not allow to distinguish the classifier (algorithm) from the population (nature).  Clearly, in this setting a constant threshold $c$ applied globally to $X$ is generally insufficient necessitating a more nuanced approach. 
	
	This paper investigates that novel approach resulting in an interpretable methodology for adapting the threshold with respect to the covariate $Z$. This approach yields easily interpretable rules and can be used to improve the accuracy of decisions.  Our core contributions are as follows. Under mild conditions, there exists a measurable threshold function $c(\cdot)$ such that the resulting threshold rule matches the conditional probability $\pi(z)$. Given a training sample $ (Y_1, Z_1), \ldots, (Y_n, Z_n) $,  a simple and computationally cheap nonparametric sliding-window estimator $\hat{c}_n(\cdot)$ is proposed. Asymptotic properties of $\hat{c}_n(\cdot)$ are established including non-asymptotic error bounds through concentration inequalities,  thus providing performance guarantees for finite sample sizes and allowing to examine the convergence rate. A central limit theorem is provided enabling the construction of  pointwise asymptotic confidence intervals. Going beyond this, we view the proposed estimators as functional estimators and provide some empirical process theory. Specifically,  functional sequential central limit theorems are shown, which allow to construct statistical tests and uniform confidence band for the functional estimator $ \{ \hat{c}_n(z) : z \in K \} $ over suitable subsets $K \subset \R$, and to devise sequential inference such as testing for the presence of a change-point. 
	
	Threshold adaptation of such rules has been introduced by \cite{Steland2024} for discrete $Z$. Motivated by the fact that often a detector should be more sensitive for certain regions of the $z$-sample space, e.g., since they represent risky cases, threshold functions were identified and proposed which guarantee  type I error (false alarm) rate, and simultaneously achieve increased sensitivity for certain classes defined by $Z$.  \cite{steland2026adaptive} extends the theoretical results by allowing for unknown $ \Psi $ of the $ X_i$'s, and studies a nonparametric estimator based on the sample quantile function. Contrary to the intention of these works, in the present paper the threshold function is used to align the conditional alarm rate to a reference population resp. sample.  

	The practical utility of our approach is illustrated through a comprehensive analysis of the FICO credit loan dataset, a well known benchmark in the field of explainable (resp. interpretable) machine learning. Our findings demonstrate that the proposed interpretable rule, which adaptively thresholds the FICO score by a function of the loan amount, achieves an in-sample accuracy of $72.3\%$ and out-of sample $ 72.2\%$. This performance is remarkably competitive with that of state-of-the-art classifiers including transformer networks, which are examined as an additional competitor to the neural networks studied in the literature. These competitors often rely on substantially more involved and less transparent decision rules. Moreover, the accuracy of the proposed Gaussian approximations are examined by a small simulation study.
	
	The rest of this paper is organized as follows. Section 2 provides preliminaries, identifies the ground-truth threshold function and briefly discusses fairness properties of the associated rule. Section 3 introduces the nonparametric estimator $\hat{c}_n(\cdot)$ and presents its key asymptotic properties, including nonasymptotic error bounds,  (functional) central limit theorems and confidence intervals resp. uniform confidence bounds. In Section 4 a local bandwidth selection method is proposed based on a Lepski-type approach.  Applications of the theoretical results are discussed in Section 5. Section 6 details the application to the FICO credit loan dataset and discusses the empirical results.

\section{Preliminaries and method}

\subsection{Conditional calibrated thresholding is always possible}

As explained in the introduction, the goal is to select the threshold function in such a way that the resulting adapted-threshold classifier maintains the conditional alarm probabilities. 

To facilitate the subsequent analysis, we standardize the feature $X$. Let $U = (X - \mu) / \sigma$, where $\mu = E(X)$ and $\sigma^2 = \Var(X)$ represent the marginal mean and variance of $X$, respectively. We assume that $U$ possesses a continuous distribution with a strictly increasing cumulative distribution function (CDF) $\Psi $ with density $ \psi$. The regressor $Z$ is assumed to be independent of $U$ with distribution $P_Z$ on $\mathbb R$ with support $\mathcal Z=\operatorname{supp}(P_Z)$. Indeed, no absolute continuity of $P_Z$ is required. The independence assumption is made for clarity of presentation and can be relaxed, see Remark~\ref{CondDF}.

Let us first consider the case that  the true alarm event $ \{ Y = 1 \} $ coincides with the event $ \{ X > \mu + \sigma c(Z) \} = \{ U > c(Z)\}$. This means, the rule $ U > c(Z) $ is an equivalent reformulation of the true signal event. Specifically, this holds true in a classical monitoring settting, where one sets up the monitoring rule signalling an alarm when $ U > c(Z) $ and thus {\em defines} the alarm indicator $ Y $ by $ \eins_{\{U > c(Z) \}} $. Denote by
\[
	\pi(z) = P(Y_1 = 1| Z=z)
\]
the conditional alarm probability given $ z \in \calZ $, and the true marginal alarm rate by \[ 	\alpha^* = \int \pi(z) \, dP_Z(dz). \]  
  The equation
\[
  1-\Psi(c(z)) = P(U>c(Z) | Z=z ) = \pi(z)
\]
easily leads to the solution
\[
  c_0(z) = \Psi^{-1}(1-\pi(z)), \qquad z \in \calZ.
\]

A key distinguishing feature of our proposed framework is that it does not necessitate the restrictive assumption that the true event $ E = \{Y=1\}$ is defined in terms of a threshold rule based on $X$. Our setting is more general: $Y$ may be an arbitrary binary outcome, whose underlying generative mechanism  is unknown to us. We use $X$ (resp. $U$) as a measurable predictor to infer $Y$ and leverage the threshold function $ c(Z) $ as a kind of degree of freedom. 

The following simple lemma establishes the existence of a ground-truth  threshold function: Under a mild condition, a unique function $c(\cdot)$ can always be found that aligns the conditional probability that $U$ exceeds $ c(Z) $ with the true conditional probability $\pi(z)$.

\begin{lemma} (Alignment Threshold)
\label{TheLemma}
	 Suppose that $ \Psi $ is a continuous and strictly increasing c.d.f.. Then there exists a function $ c : \calZ \to \overline{\R} $ such that 
\[ 
	\pi(z) = P( U_t > c(Z_t) | Z_t = z ),
\] 
$ P_Z $-almost surely, namely
\[
    c(z) = (1-\Psi)^{-1}( \pi(z) )  = \Psi^{-1}(1-\pi(z)), \qquad z \in \calZ.
\]
The resulting rule $ \eins_{\{ X > \mu + \sigma c(Z)\}} $ has the alarm rate $ \alpha = P( X > \mu + \sigma  c(Z) ) = \alpha^* $.
\end{lemma}

\begin{proof} Define $ h(z) = P( U_t > c(Z_t) | Z_t = z ) = 1- \Psi( c(z))$.
	Let $ c(z) = \Psi^{-1}( 1-\pi(z) ) $. Then
	\[
	  P( U_t > c(Z_t) | Z_t = z ) = 1 - \Psi( c(z) ) = 1 - (1-\pi(z)) = \pi(z).
	\]
   By definition of $ c(z) $, $ 1 - \Psi(c(z))  = (1-\Psi)(c(z) ) = \pi(z)$. Hence,  the alternative formula 
  $ c(z) = (1-\Psi)^{-1}(\pi(z)) $ follows.
\end{proof}

\begin{remark}
\label{CondDF}
	If the conditional d.f. $ \Psi_z(x) = P( U_t \le x | Z_t = z ) $, of $ U_t $ given $ Z_t = z $ depends on $z$, then the formula $ c(z) = \Psi_z^{-1}(1-\pi(z)) $ follows. Provided the functions $ \Psi_z(x) $, $ z \in \calZ $, are known to us, the methodology easily carries over. Thus, for simplicity of presentation, we confine our discussion to the case $ \Psi_z = \Psi $.
\end{remark}

The above lemma demonstrates that, for {\em any} given true conditional probability function $\pi(z)$, it is {\em always} possible to define a threshold function $c(z)$ of the feature $Z$ such that the conditional probability of $U$ exceeding this threshold precisely matches $\pi(z)$. This means, the adaptive threshold $c(Z)$ acts as a sufficient "degree of freedom" to achieve perfect calibration of the indicator $\mathbf{1}_{\{U > c(Z)\}}$ with respect to the true conditional probabilities $P(Y=1|Z=z)$.

Crucially, this result holds irrespective of whether the true binary outcome $Y$ is actually generated by a simple thresholding rule on $X$ (or $U$). In many real-world applications, $Y$ may be the result of a complex, unobserved data-generating process that is not directly expressible as $X$ exceeding a threshold. The lemma asserts that, even in such intricate scenarios, we can construct an observable thresholding rule $\mathbf{1}_{\{U > c(Z)\}}$ whose conditional probabilities {\em match} those of $Y$. 

While the lemma guarantees the precise matching of conditional probabilities, it is important to distinguish this from the exact coincidence of individual classification decisions. That is, the labels $\mathbf{1}_{\{U_t > c(Z_t)\}}$ assigned by the thresholding rule may not perfectly align with the observed outcomes $Y_t$, they only match on average. However, if $X$ is a strong predictor of $Y$, one can expect that the  adaptive threshold rule $\mathbf{1}_{\{U > c(Z)\}}$ will exhibit high classification accuracy. Our real data example illustrates this.

\subsection{Threshold estimation and decision rules}

To estimate the threshold function we assume that we are given a learning sample 
\[ (Y_1,U_1,Z_1), \ldots, (Y_n,U_n,Z_n) \] 
of size $n$ distributed as  $ (Y,U,Z) $. Here, it is implicitly  assumed that $ U_t = (X_t - \mu)/ \sigma $ and $ Z_t $ are observable. 

In view of Lemma~\ref{TheLemma}, we can estimate $ c(z) $ by plugging in an estimator of  $ \pi(z) $. Notice that $ \pi(z) $ is a conditional probability given $ Z = z $. We fix some suitably chosen $h>0 $ and estimate 
\[ 
	 \pi(z,h) = P( Y=1 | Z \in [z-h, z+h]) = \frac{\int_{z-h}^{z+h} \pi(z) \, d P_Z(dz)}{ P_Z( (z-h,z+h] )} 
\] 
and in turn $ c(z,h) = \Psi^{-1}(1-\pi(z,h) ) $ by a computationally efficient sliding window averaging procedure. The conditional alarm probability $ \pi(z,h) $ is estimated by 
\[
\hat \pi_n(z,h) 
= \left\{ \begin{array}{ll} \frac{ A_n(z)}{ B_n(z)}, & \qquad  B_n(z) > 0, \\
	0 & \qquad  B_n(z) = 0. \end{array} \right.
\] 
and the threshold function $ c(z) $ by
\[
	\hat c_n( z ) = \Psi^{-1}(1-\hat\pi_n(z,h) )
\]
for each $ z $ with $ B_n(z) > 0 $.   Here,
\begin{align*}
	A_n(z,h) & = \#(1 \le i \le n : Y_i=1, |Z_i-z|\le h) /n, \\
	B_n(z,h) & =  \#(1 \le i \le n:|Z_i-z|\le h)/n,
\end{align*}
whose expectations are 
\begin{align*} 
	p_A(z,h) &= P(Y_1=1,|Z_1-z| \le h ), \\
	p_B(z,h) &= P( |Z_1-z|\le h ),
\end{align*} 
such that $ \pi(z,h) = \frac{p_A(z)}{p_B(z)} $.

If both $ X_t $ and $ Z_t $ are observed, then a generic new observation $ x $ with covariate $ z $ is classfied as suspicious, if
\[
x > \hat{\mu}_n + \hat{\sigma}_n \hat c_n( z ),
\]
where $ \hat{\mu}_n $ is the sample average and $ \hat{\sigma}_n $ the sample standard deviation of the observations $ X_1, \ldots, X_n $. Further, the in-sample observation $ X_i $ is marked suspicious, if
\[
X_i >  \hat{\mu}_n + \hat{\sigma}_n \hat c_n( Z_i ),
\]
for $ 1 \le i \le n $.

Most of our analysis is for fixed (but small) $ h>0 $ and thus holds for arbitrary laws $ P_Z $ of $ Z$. If $z\in\operatorname{supp}(P_Z)$ and $\pi$ is continuous at $z$,
then
\[
\pi(z,h)\longrightarrow\pi(z)
\qquad\text{as }h\downarrow0.
\]
Indeed,
\[
|\pi(z,h)-\pi(z)|
\le
\sup_{\substack{u\in\operatorname{supp}(P_Z)\\|u-z|\le h}}
|\pi(u)-\pi(z)|.
\]
Moreover, if $\pi$ is uniformly continuous on a neighborhood of a compact set $K$, the convergence is uniform over $K$.

\begin{remark}
	Both $ A_n(z,h) $ and $ B_n(z,h) $ can be calculated recursively and thus without access to the  whole underlying data set. For example, $ A_1(z) = \eins_{\{Y_1=1, |Z_1-z| \le h)\}} $ and
	\[ 
	A_n(z) = \left((n-1) A_{n-1}(z) + \eins_{\{ Y_n = 1, |Z_n-z| \le h\}} \right)/n, \qquad n > 1.
	\] 
	Therefore, computing these quantities on a grid of $ z$-values of a fixed size, say, $L$, is feasible for arbitrary sample sizes $n$. In practice, one may use an equidistant grid $ z_l = \min \calZ + (l-1) h $, $ 1 \le l \le L $, with $ L = \lceil (\max \calZ - \min \calZ)/h \rceil +1 $. 
\end{remark}

\begin{remark}
	$ \hat\pi_n(z,h) $ is a Nadaraya-Watson estimator using a rectangular kernel, which is commonly used to estimate conditional means under the asymptotic regime $ h \to 0 $ and $ nh \to \infty $. However, we apply it to binary $ Y_i$ and, as explained below, consider the asymptotc regime $ n \to \infty $ with fixed $h$.  
\end{remark}

\begin{remark}
	For simplicity of presentation, we confine ourselves to a univariate variable $Z$, but the generalization to a random vector is straightforward by replacing $ | \cdot | $ in the above defintions by a vector norm. 
\end{remark}

\subsection{Representative fairness by calibration}
\label{sec:representative-fairness}

As already briefly indicated in the introduction, the property  $P(U > c(Z)|Z=z) = \pi(z)$ can be interpreted as a certain notion of algorithmic  fairness. To give a formal definition denote the underlying probability space on which all random variables are defined by $ (\Omega, \calF, P) $ and $ X $ maps to a measurable space $ (\calX, \calF_\calX) $.

\begin{definition} Let $ Y \in \{ 0,1 \} $ and $ X, Z $ be random variables defined on $ (\Omega, \calF, P) $ and $ \hat{Y} = \eins_{\{ X \in A\}} $, $ A \in \calF_\calX $, be a classifier. The classifier satisfies the property of {\em representative fairness} with respect to {\em matching variable} $ Z $, if $ P_{\hat{Y} | Z} = P_{Y|Z} $.
\end{definition}

This notion of fairness differs from the notions of equalized odds (EOD), equal opportunity (EO) or demographic parity (DP), \cite{PessachShmueli2022}. Especially, representative fairness does not ensure  for a protective variable $Z$ with relevant outcomes $ E^+ $ and $E^-$, say, $ E^+ = \{ Z > 0 \} $ and $ E^- = \{ Z \le 0 \} $,  that the chances are equalized, $ P(\hat  Y = 1| E^+ ) = P(\hat Y = 1 | E^-) $, or that the true resp. false positive rates are equalized. This is in contrast to DP, EOD and RO. DP requires that the conditional label distributions of the classifier given a protective variable, say $A$, are equalized or at least very close. EOD requires that the true positive rate (TPR) and the true negative rate (TNR) of the classifier are equalized (or very close) for given $A$, whereas EO focuses on the TPR, \cite{HardtEtAl2016}. 

Whereas these measures of fairness look at statistical quantities used to evaluate a classifier for different values of a protective variable, representative fairness considers the distribution of the classifier with respect to a matching variable (which usually is not a protective variable) and does not condition on $Y$. It requires that the distribution of the classifier labels is equal to the distribution of the ground-truth labels given the matching variable $Z$. In this way, a clustering of positive (or negative) labels for certain outcomes of $Z$ is avoided. One cannot infer from the label distribution across the range of $Z$ whether the labels are the true ones (from nature) or come from the classifier. 

Representative fairness is a strong property, as it guarantees that the classifier does not change the label distribution within classes of $Z$. But it also implies that a bias contained in $Y$ is transferred to $ \hat{Y} $. Specifically, if $Z$ is finer than a protective variable $A$ in the sense that $ A = k(Z) $ for some function $k$, then the tower property of conditional expectations implies that the conditional law $ Y | A $ is 'copied' to the classifier. This can be problematic, if historical bias is present in the data. But by selecting a (curated) random sample without historical bias and ensuring acceptable fairness measures, one can circumvent this issue. 

\textbf{Low-probability events:} As well known, classifiers tend to have difficulties with respect to fairness in low-probability regions, especially if they are not flexible enough to adapt to sparsely populated regions. Since low-probability regions of $Z$ tend to be low-probability regions of the classifier, as explained below, the proposed threshold adaptation mitigates this effect by 'copying' $ \pi(z) $ into the classifier rule.

To briefly discuss the inheritance of low-probability regions, assume the classifier is given by some (measurable) function $ g( \xi  $) for regressors $ \xi $. A low-probability region  of the $ \xi $-space will be  sparsely represented in a training sample with high probability. This can result in unfair performance, since the objective criterion used to train the classifier may be optimized by ignoring such a region and assigning the label of nearby regions of higher probability. To explain the issue that low-probability regions of the $ Z$-sampling space $ \calZ $ may suffer from a classifier which has issues with low-probability regions of its $\xi $-sampling space $\R^{d} $, let us assume that $ \xi = (Z,\xi') $ with $ \xi' \in \R^{d-1} $ and $ Z \in \calZ \subset \R$. For a set $ D \subset \calZ \times \R^{d-1} $ denote by $ \Pi_1 D = \{ Z \in \calZ : (Z,\tilde \xi') \in D \ \text{for some $ \tilde \xi' \in \R^{d-1}$ } \} $ the projection of $D$ onto the first coordinate. Sets whose projection is a subset of a small probability set of $Z$ are low-probability sets for a classifier $g(\xi) $ taking values in $ [0,1]$. A $ \varepsilon $-probability event $A$ for a random variable $ X $ is a measurable set $A$ such that $ P(X \in A) \le \varepsilon $. We have the following simple result:

\begin{lemma}
	\label{SmallProbEvents}
	Let $ \varepsilon> 0 $ and  $ A $ be a $ \varepsilon $-probability region for $ Z $, i.e. a measurable subset of $ \calZ $ such that $ P(Z \in A) \le \varepsilon $.  Then any measurable set $ B \subset [0,1]$ with $ \Pi_1 g^{-1}(B) \subset A $ is a $ \varepsilon $-probability region for $ g(\xi ) $. 
\end{lemma}

\begin{proof} This follows from
	$
	\{ g(\xi) \in B \} = \{ (Z,\xi')  \in g^{-1}(B)\} \subset \{ Z \in \Pi_1 g^{-1}(B) \} \subset A.
	$
\end{proof}

\subsection{Relation to probabilistic calibration by histogram binning}

The proposed local probability estimator is somewhat related to probabilistic calibration by histogram binning, see \cite{ZadroznyElkan2001,GuoEtAl2017}, in that both methods average binary outcomes over a neighborhood.  But histogram calibration partitions
the range of an existing prediction score, $S \in [0,1]$, and returns a recalibrated probability to ensure approximately $\mathbb{E}[Y\mid S]=S$, whereas our procedure smooths over a contextual variable
and subsequently converts the estimated label frequency into a
context-dependent decision threshold. Its objective is conditional
label-distribution matching, $\mathbb{E}[\widehat Y-Y\mid Z]=0$, rather than probability calibration, $\mathbb{E}[Y\mid S]=S$. If $ Z = S $ and identical disjoint bins are used, the probability estimates coincide, but the distinction still lies in the subsequent threshold construction and its statistical analysis.

\section{Asymptotic theory}

We begin in Section~\ref{Sec:NonAsympBounds} with nonasymptotic uncertainty error bounds for the estimated conditional probability and the induced estimated threshold. Although the former is a quite standard problem, the results provided here seem to be new. Section~\ref{Sec:AsNormality} provides a central limit theorem (CLT) and uncertainty quantification by confidence intervals. In Section~\ref{Sec:EmpProcessTheory} we study the estimators as functional ones. By using empirical process theory weak convergence results are obtained which combined with estimation of the asymptotic covariance structure eventually allow to simulate asymptotic laws.  That theory fixes the bandwidth $h$, since otherwise no non-trivial weak limit exists.

 \subsection{Nonasymptotic error bounds and consistency}
 \label{Sec:NonAsympBounds}

The nonasymptotic high probability concentration bounds, see, e.g., \cite{BBoucheron2013} for background, derived in this section provide uncertainty quantification for any sample size and the convergence rate. They allow to study the case $ h \to 0 $ as well. It turns out that one can make use of Okamoto's exponential inequality for the tail of a binomially distributed random variable $ Y $ with sample size $n$ and success probability $p \in (0,1) $,
\[
	 \tau(p,t) = P(|Y/n - p| \ge t ), \qquad t > 0.
\]
These bounds are sharper than the well known Hoeffding bound $ \tau(p,t) \le 2 \exp(-2nt^2) $, if $ p < 1/8 $. But, since the resulting bounds turn out to be more involved, we provide the results using Hoeffding's inequality, too.

 \begin{theorem} 
\label{NonAsympPi}
 	 Fix $ z \in \calZ $ and $h> 0$  with $ 0 < \pi(z,h) < 1/8 $. 
 	 Then, for any $ t > 0 $ and all $ n \ge 1 $ the following bounds hold with probability at least $ 1 - \delta $:
\begin{itemize}
\item[(i)] Firstly, 
\[
\resizebox{\linewidth}{!}{$\displaystyle | \hat\pi_n(z,h)  -  \pi(z,h)  | \le b_1(t) = 2 \sqrt{\pi(z,h)} \sqrt{- \log\left( 1 - \frac{\log( 2/\delta) }{n \pi(z,h)}  \right)} - \log\left( 1 - \frac{\log( 2/\delta) }{n \pi(z,h)}  \right).$}
\]
\item[(ii)] Secondly, 
\[ | \hat\pi_n(z,h)  -  \pi(z,h)  | \le b_2(t) = 2 \sqrt{\pi(z,h)} t + t^2 \]
with $ t = t(\delta) $ given by
	\[
\resizebox{\linewidth}{!}{$\displaystyle t = \sqrt{ - \log\left(  \frac{1}{p_B(z,h)} \left[  \left\{ \frac{\delta}{2} (1-(1-p_B(z,h))^n) + (1-p_B(z,h))^n \right\}^{1/n} - (1-p_B(z,h)) \right] \right) }.$}
\]
\item[(iii)] Lastly,
\[
	| \hat\pi_n(z,h)  -  \pi(z,h)  | \le b_3(t) = \frac{t}{\sqrt{2}} 
\]
where  $ t = t(\delta) $ is as in (ii).
\end{itemize}
\end{theorem}
 
Figure~\ref{FigBounds} illustrates that the second and third more involved bounds, which are based on an exact calculation of the moment generating function, $ \varphi_N$, of the random sample size $N = n B_n(z,h)$ and its inversion. 
The  bound $ b_2(t) $ uses Okamoto's inequality, whereas $ b_3(t) $ makes use of Hoeffding's inequality. Both bounds are much tighter than $ b_1(t)$, which avoids exact inversion of $ \varphi_N$ by using an exponential bound. 
 
 \begin{figure}
 	\begin{center}
 	\includegraphics[width=10cm]{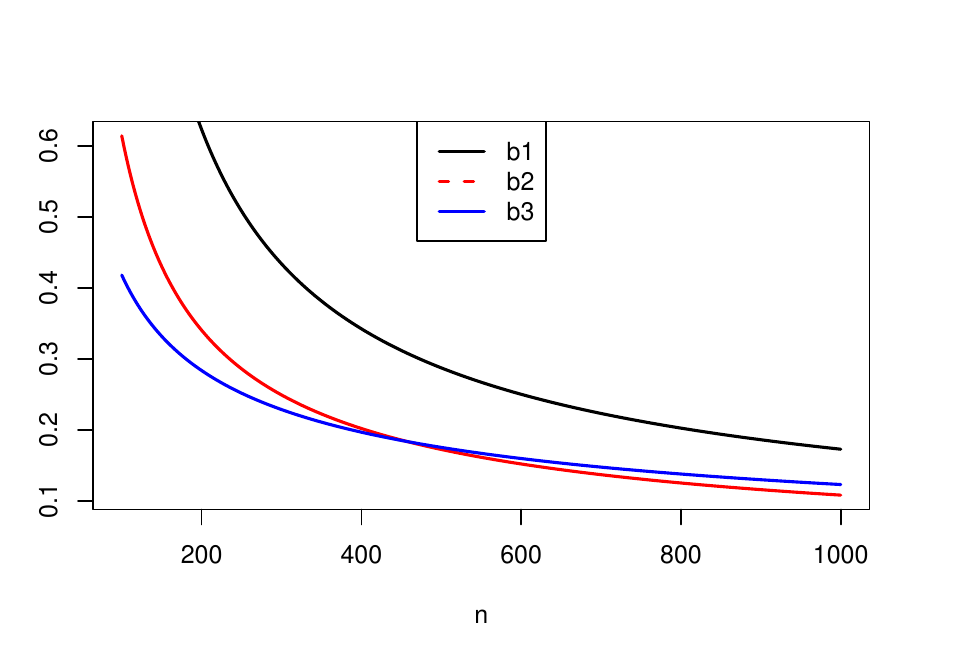}
 	\end{center}
 	\caption{Comparison of Okamoto's and Hoeffding's bound for fixed $p_B(z) = 0.1 $, $ \pi(z,h) = 0.05 $ and $ \delta = 0.1$, as a function of the sample size $ n \in \{ 100, \ldots, 1000\} $.}
 	\label{FigBounds}
 \end{figure}

Reformulating the above concentration bounds provides confidence intervals
\[
  \hat \pi_n(z,h) - b_i( t( \delta) ) \le \pi(z,h) \le \hat \pi_n(z,h) + b_i( t( \delta) ), \qquad i = 1, 2, 3,
\]
with confidence level at least $ 1-\delta $. By monotonicity, these intervals directly translate to confidence intervals
\[
(1-\Psi)^{-1}( \hat \pi_n(z,h) + b_i( t( \delta) ) ) \le c(z,h) \le (1-\Psi)^{-1}( \hat \pi_n(z,h) - b_i( t( \delta) ) ), \qquad i = 1, 2, 3,
\]
for the treshold $ c(z,h) $.

Using $ -\log(1-x) \le x + x^2 $ for $ 0 < x < 1/2 $, the above nonasymptotic bounds yield the high-probability convergence rate $ O( 1/\sqrt{n p_B(z,h)}) $.  Especially,  consistency for $ \pi(z,h) $ holds as long as $ n p_B(z,h) \to \infty $. Pointwise consistency along a sequence $h_n\downarrow0$ requires
\[
n\,p_B(z,h_n)\longrightarrow\infty,
\]
together with a condition ensuring $\pi(z,h_n)\to\pi(z)$, for example continuity of $\pi$ at $z$.

Notice that the  condition $ n p_B(z,h) \to \infty $ includes settings with $ h \to 0 $ at a rate depending on the rate at which $ p_B(z,h) $ tends to $0$ as $ h \to 0 $. In other words, the {\em pointwise anti-concentration of $Z$},  $ p_B(z,h) = P( |Z-z| \le h ) $, and its supremum and infimum,
\[
  a_{P_Z}^+(h) =  \sup_{z \in K} p_B(z,h),  \quad  a_{P_Z}^-(h) =  \inf_{z \in K} p_B(z,h) \qquad h > 0,
\]
naturally determine the convergence rate.  For example, if $ Z_1 \sim N(0,1) $, then $ p_B(z,h) = 2h / \sqrt{2\pi} e^{-z^2/2} (1+ O(h^2)) $, and thus we have the anti-concentration inequalies
\[
a_{\calN(0,1)}^+(h) = \sup_{z \in K} p_B(z,h) \le \frac{2h}{\sqrt{2 \pi}} (1+ O(h^2) ),
\]
and, by compactness of $K$, 
\[
	a_{\calN(0,1)} ^-(h) = \inf_{z \in K} p_B(z,h)  \ge C_1 h (1+ C_2 h^2)
\]
for constants $ C_1,  C_2 $ depending on $K$.  More generally, $ a_{P_Z}^+(h) $ and $ a_{P_Z}^-(h) $ are $O(h) $,  if $ Z_1 $ attains a bounded and positive density on $K$. In such cases, consistency holds under the usual smoothing condition $ n h \to \infty $. 

If we suppose that for some $\beta>0$,
\[
P_Z([z-h,z+h])\asymp h^\beta
\]
at a point $z$ under consideration. consistency holds provided
\[
nh_n^\beta\to\infty.
\]
Here, for an atom $P_Z(\{z\})>0$, the local mass does not tend to zero,
and the effective rate remains of order $n^{-1/2}$.

Specifically, if $ Z_1 $ attains a density $f(x) $ with singularities, then such different convergence arise. For example, consider the density $ f_\gamma(x) = 0.5 (1-\gamma) |x|^{-\gamma} \eins_{0<|x|<1} $, $x \in \R $, for some $ 0 < \gamma < 1 $. Then $ p_B(0,h) = h^{1-\gamma} $, i.e., $ \beta = 1-\gamma$, leading to the condition $ n h^{1-\gamma} \to \infty $.

\subsection{Asymptotic normality and confidence interval}
\label{Sec:AsNormality}

To obtain tight uncertainty intervals central limit theorems are a suitable approach. We confine ourselves to the case that the bandwith $h$ is held fixed and the sample size $n$ approaches $ \infty $, because, firstly, in applications such as anomaly detection the alarm rates $ P(Y_1=1) $ and $ P(Y_1=1|Z=z) $ are often rather  small, and, secondly, we are interested in a weak convergence result of the empirical process associated to the estimators, which prevents $ h \to 0 $, see the discussion in the next section.  For these reasons, we assume that the bandwith $h$ is held fixed and consider the asymptotics as $ n \to \infty $. The CLTs provided here for fixed $z$ can be used to calculate pointwise confidence intervals and also prepares our treatment of the process versions.

Denote $ z_{q} $ the $q$-quantile of the standard normal law. 

\begin{theorem} 
	\label{CLTpi}
	Fix $ z \in \calZ $ with $ \pi(z,h) \in (0,1) $ and $ p_B(z,h) > 0 $.
	\begin{itemize}
		\item[(i)] It holds
		\[ 
		\sqrt{n}( \hat\pi_n(z,h) - \pi(z,h) ) \stackrel{d}{\to} N( 0, \eta^2 ),
		\]
		as $ n \to \infty $, where
		\[
		\eta^2(z,h) = \frac{\pi(z,h)(1-\pi(z,h))}{p_B(z,h)}.
		\]
		The plug-in estimator $ \hat\eta^2_n(z,h) $ obtained by replacing $ p_A(z,h) $ and $ p_B(z,h) $ by $ A_n(z,h) $ and $ B_n(z,h) $ is consistent. 
		\item[(ii)]
		For $ \alpha \in (0,1) $ an asymptotic $(1-\alpha) $-confidence interval of $ \pi(z,h) $ is given by
		\[
		CI_\alpha( \pi)  = \left[  \hat\pi_n(z,h) - z_{1-\alpha/2} \frac{\hat\eta_n}{\sqrt{n}}, \hat\pi_n(z,h) + z_{1-\alpha/2} \frac{\hat\eta_n}{\sqrt{n}}\right].
		\]
	\end{itemize}
\end{theorem}

The CLT for $ \hat{c}_n(z,h) $ and an associated confidence interval follows by an additional Taylor expansion.

\begin{corollary} Fix $ z \in \calZ $ with $ \pi(z,h) \in (0,1) $, $ p_B(z,h) > 0 $ and  $ \psi( \Psi^{-1}(1-\pi(z,h))) > 0$. 
	\begin{itemize}
		\item[(i)]  We have
		\[
		  \sqrt{n}( \hat c_n(z,h) - c(z,h) ) \stackrel{d}{\to} N(0, \zeta^2 ),
		\]
		as $ n \to \infty $, where $ \zeta^2(z,h) =  \frac{\eta^2(z,h)}{\psi^2(\Psi^{-1}(1-\pi(z,h)))} $.  \\ The plug-in estimator $ \hat\zeta_n^2(z,h) =  \frac{\hat\eta_n^2(z,h)}{\psi^2(\Psi^{-1}(1-\hat\pi_n(z,h)))}$ is consistent. 
		\item[(ii)] For $ \alpha \in (0,1)$ an asymptotic $ (1-\alpha )$-confidence interval of $ c(z,h) $ is given by
		\[
		CI_\alpha( c ) =   \left[ \hat{c}_n(z,h) - z_{1-\alpha/2} \frac{\hat \zeta_n}{ \sqrt{n}}, \hat{c}_n(z,h) + z_{1-\alpha/2} \frac{\hat \zeta_n}{ \sqrt{n}} \right].
		\]
	\end{itemize}
\end{corollary}

\subsection{Empirical process theory}
\label{Sec:EmpProcessTheory}

Interpreting the estimators $ \hat \pi_n(\cdot,h) $ and $ \hat c_n(\cdot,h) $ as functional estimators on a suitably chosen domain $ K \subset \R $ and thus studying them as empirical processes is more intricate. But, from a theoretical perspective, it is interesting in its own right, and it allows for functional inference. We discuss some applications in the next section. 

Before proceeding, it is crucial to notice that  $ \hat \pi_n(z, h) $ behaves as a kernel density estimator for which it is known that the associated empirical process cannot converge weakly to a non-degenerate process under the  asymptotic regime $ h \to 0 $ and $ nh \to \infty $. Indeed, although the finite dimensional distributions do converge to non-degenerate multivariate Gaussian distributions, any weak limit process needs to be almost everywhere $0$, see \cite{Nishiyama01032011} for weak convergence in $L_p$ spaces and \cite{Stupfler2016} for the space $ l^\infty  $ of bounded functions equipped with the supnorm used here. The degeneracy of any weak limit carries over to our setting, since the numerator $ A_n(z,h)  $ of $ \hat \pi_n(z,h) $ is a Nadaraya-Watson type estimator. The following impossibility result generalizes \cite{Nishiyama01032011} and is proved in Appendix A.

\begin{theorem} 
\label{IP-impossible}
	Let $ (X_i,Y_i) \sim f $, $ i \ge 1 $, be an i.i.d. sequence with $ E(Y_1^2) < \infty $.  Define $ \hat m_n(x,h) = \frac{1}{nh} \sum_{i=1}^n \eins_{\{ | x-X_i | \le  h \}} Y_i$ and $m_n(z,h) = E( \hat{m}_n(x,h) )$. If $ M_n(x,h) = \sqrt{nh}( \hat m_n(\cdot,h) - m_n(\cdot,h) ) $ converges weakly   in $ L_2(\R) $ to some random process $ G(\cdot) $, as $ h \to 0 $ and $ nh \to \infty $, then
	$ G = 0$ in $ L_2(\R)$. 
\end{theorem}

To circumvent this issue, we study weak convergence  for $ n \to \infty $ and fixed $ h$.  Recall at this point that a class  $ \calF $ 
of functions $ \R \times \{0,1 \} \to \R $ is called a GC (Glivenko-Cantelli) class, if
\[
  (P_n-P)(f) = \frac{1}{n} \sum_{i=1}^n [ f(Y_i,Z_i) - E f(Y_1,Z_1)], \qquad f \in \calF,
\]
satisfies the uniform strong law of large numbers, i.e.
\[
  \lim_{n \to \infty} \sup_{f \in \calF}  |(P_n-P)(f) |= 0, \qquad \text{almost surely.}
\]
Here, $ P_n = \frac{1}{n} \sum_{i=1}^n \delta_{(Y_i,Z_i)} $ is the empirical measure associated to the sample $ (Y_i,Z_i), 1 \le i \le n $,  $ \delta_x $ is the Dirac measure in a point $x$, $ Q(f) :=  \int f \, d Q $ for a probability measure $Q$ on $ (\Omega,\calF) $ and $Q^* $ denotes outer probability. Moreover,  $ \calF $ is called a $ P$-Donsker class, if under $P$ the empirical process $ \sqrt{n}(P_n-P)(f)  $, $ f \in \calF $, satisfies a functional central limit theorem (invariance principle),
\[
  \sqrt{n}(P_n-P)(f) \Rightarrow B(f)
\]
for a centered tight $P$-Brownian bridge process $ \{ B(f) : f \in \calF \} $ with covariance function $ \Cov( B(f), B(g) ) = \int f(x) g(x) \, dP(x) - \int f(x) \, d P(x) \int g(x) \, d P(x) $, $ f, g \in \calF $. Here $ \Rightarrow $ signifies weak convergence in the space $ l^\infty(\calF;\R) $ in the sense of \cite{VaartWellner2023}. 

We are interested in establishing weak convergence of the empirical process
\[
  \mathscr{C}_n(z,h) = \sqrt{n}\left( \hat c_n(z,h) - c(z,h) \right), \qquad z \in K,
\]
and its sequential generalization
\[
  \mathscr{C}_n(t,z,h) = \sqrt{n} \left( \hat c_{\lfloor nt \rfloor} (z,h) - c(z,h) \right), \qquad z \in K, t \in [t_0,1],
\]
for some $ t_0 \in (0,1) $ and suitably chosen compact index sets $ K \subset \R $.  This will follow, under suitable assumptions, easily from the weak convergence of 
\[
  \mathscr{P}_n(z,h) = \sqrt{n} \left(  \hat \pi_n(z,h) - \pi(z,h) \right), \qquad z \in K,
\]
and
\[
  \mathscr{P}_n(t,z,h) = \sqrt{n} \left(  \hat \pi_{\lfloor nt \rfloor }(z,h) - \pi(z,h) \right), \qquad z \in K, t \in [t_0, 1], 
\]
respectively. The sequential processes are indexed by $ [t_0, 1] \times  K $ and thus weak convergence is then studied in the space $ l^\infty( [t_0,1] \times K; \R ) $.

\begin{theorem} 
\label{DonskerCLT}
	\begin{itemize}
		\item[(i)]
	Let $ K \subset \R $ be a compact set with $ \pi(z,h) \in (0,1) $ and $ 0 < \inf_{z \in K}  p_B(z,h) $. Then 
	\[
	  \mathscr{P}_n(\cdot,h)  \Rightarrow  \mathscr{P}(\cdot, h), \qquad n \to \infty,
	\]
	in $ \ell^\infty(K) $, 
	for some tight centered Gaussian process $ \mathscr{P} = \{ \mathscr{P}(z,h) : z \in K \}$ with covariance function
	\[
\resizebox{\linewidth}{!}{$\displaystyle \eta(u,v,h) =	\frac{1}{p_B(u,h)p_B(v,h)} \left( (1-\pi(u,h) - \pi(v,h)) \gamma_2(u,v,h)  + \pi(u,h) \pi(v,h) \gamma_1(u,v,h) \right),$}
\]
	 where 
	\begin{align*}
		\gamma_1(u,v,h) &= P(|Z_1-u|\le h, |Z_1-v|\le h), \\
		\gamma_2(u,v,h) &= P( |Z_1-u|\le h, |Z_1-v|\le h, Y_1=1),
	\end{align*}
	for $ u, v \in K $. More generally, the sequential empirical process converges weakly,
	\[
	  \{ \mathscr{P}_n(t,z,h) : t \in [t_0,1], z \in K \} \Rightarrow \{ \mathscr{P}(t,z,h) : t \in [t_0,1], z \in K \},
	\]
	as $ n \to \infty $, in $ \ell^\infty([t_0,1]\times K) $, for some tight centered Gaussian process $\mathscr{P}$ indexed by $ [t_0,1] \times K $ with covariance function given by
	\[
	  \Cov( \mathscr{P}(s,u,h), \mathscr{P}(t,v,h) ) = \frac{s \wedge t}{st} \eta(u,v,h) , 
	\]
	for $ s, t \in [t_0,1] $ and $ u, v \in K $.
	
	\item[(ii)] Let $K \subset \R $ compact and assume that the following conditions are satisfied.
		\begin{itemize} 
		\item[(a)] $ 0 <a \le \inf_{z \in K} \pi(z,h) < \sup_{z \in K} \pi(z,h) \le b < 1 $.
		\item[(b)]  The functions $ p_B(z,h) $ and $ p_A(z,h) $, $ z \in K $, are Lipschitz continuous,
		\item[(c)] $ q(r) = \Psi^{-1}(1-r) $ is twice differentiable on $ [a,b] $ 
	\end{itemize} 
	Then
	\[
	  \mathscr{C}_n(\cdot, h) \Rightarrow \mathscr{C}( \cdot, h), \qquad n \to \infty,  
	\]
	 in $ l^\infty(K) $, for some tight centered Gaussian process $ \mathscr{C} = \{ \mathscr{C}(z,h) : z \in K \} $ with covariance function
	\[
	\zeta(u,v,h) = \frac{\eta(u,v,h)}{\psi(\Psi^{-1}(1-\pi(u,h))\psi(\Psi^{-1}(1-\pi(v,h))}, 
	\]	
	for $ u, v \in \R $.  More generally, the sequential empirical process converges weakly
	\[
	  \{ \mathscr{C}_n(t,z,h) : t \in [t_0,1], z \in K \} \Rightarrow \{ \mathscr{C}(t,z,h) : t \in [t_0,1], z \in K \},
	\]
	as $ n \to \infty $, in $ \ell^\infty( [t_0,1] \times K) $ to some centered tight Gaussian process $ \mathscr{C} $ with covariance function given by
	\[
		\Cov( \mathscr{C}(s,u,h), \mathscr{C}(t,v,h) ) = \frac{s \wedge t}{st} \zeta(u,v,h), 
	\]
	for $ s, t \in [t_0,1] $ and $ u, v \in K $.
\end{itemize}
\end{theorem}

The covariance functions of $ \mathscr{P}(t,z,h) $ and $ \mathscr{C}(t,z,h) $  can be estimated consistently by plugging in $  \hat{p}_B(z,h), \hat{p}_A(z,h), \hat\pi_n(z,h) $ and the estimators 
\begin{align*}
	\hat\gamma_{n1}(u,v,h) &= \frac{1}{n} \sum_{i=1}^n \eins(|Z_i-u| \le h, |Z_i - v|\le h), \\
	\hat\gamma_{n2}(u,v,h) &= \frac{1}{n} \sum_{i=1}^n \eins(|Z_i-u| \le h, |Z_i - v|\le h, Y_i=1).
\end{align*}
Especially, 
\[
  \hat{\zeta}_n(u,v,h) = \frac{\hat\eta_n(u,v,h)}{\psi(\Psi^{-1}(1-\hat\pi_n(u,h))\psi(\Psi^{-1}(1-\hat\pi_n(v,h))}.
\]

\section{Local bandwidth selection with a bias allowance}
\label{sec:lepski-compact}

As a data-driven feasible appraoch to bandwidth selection, we propose to use a Lepski-type approach, \cite{chagny2016introduction}, where the largest bandwidth among candidates is selected, which is compatible in terms of accuracy with all finer estimators and minimizes the error bound. 

At this point, let us assume that  $ \pi(z) $ is Lipschitz continuous, i.e.,
\[
	|\pi(z)-\pi(z')|\le Ld(z,z'), \qquad z, z' \in \mathcal{Z},
\]
for some Lipschitz constant $L>0$.  Let 
$\mathcal H(z) = \{ h_1 < \cdots < h_J \}$ be a set of $J$ candidate bandwidths and define
\[
	N_h(z)=\{1 \le i \le n:d(Z_i,z)\le h\},\quad m_h(z)=|N_h(z)|,\quad
\widehat\pi_h(z)=\frac1{m_h(z)}\sum_{i\in N_h(z)}Y_i .
\]
We may and will assume that $ m_h > 0 $ for all $h \in \mathcal{H} $ and relevant $z$.  Within neighborhood  $N_h(z) $ the average change of $ \pi( \cdot ) $ is  
\begin{equation}
	B_h(z)=\frac L{m_h(z)}\sum_{i\in N_h(z)}d(Z_i,z)\le Lh,
	\label{eq:compact-radius}
\end{equation}
We use $B_h(z) $ as a bias allowance and add an error bound, $s_h(z) $, due to noise and thus let 
\[
	W_h(z)=s_h(z)+B_h(z).
\]
Alternatively, one use the simplified bias allowance $B_h=Lh$.  To ensure uniform error guarantee across all bandwidths, Hoeffding bound gives 
\[
  s_h=\sqrt{\log(2J/\alpha)/(2m_h)}.
\]
If $\pi(z) \le \pi_\star < 1/8$, the symmetric Okamoto bound applied to the conditional Bernoulli average instead gives 
\begin{equation}
	t_h=\sqrt{\log(2J/\alpha)/m_h},\qquad
	U_h=\min\{p_\star,(t_h+\sqrt{\widehat\pi_h+2t_h^2})^2\},
	\qquad s_h=2\sqrt{U_h}\,t_h+t_h^2.
	\label{eq:compact-okamoto}
\end{equation}

The Lepski method accepts a candidate bandwidth $h$, if the associated estimator is compatible with all finer estimates using smaller bandwidths and not yet increases the bias. Thus, let  
\begin{align}
	\label{acceptance-set}
	\widehat{\mathcal A}_n(z)
	&=\left\{h\in\mathcal H(z):
	|\widehat\pi_h(z)-\widehat\pi_g(z)|
	\le s_h(z)+s_g(z)+B_h(z)+B_g(z)
	\ \text{for all }g\le h\right\},
\end{align}
the set of accepted bandwidth candidates and define 
\begin{align}
	\widehat h_n(z)&=\mathop{\arg\min}_{h\in\widehat{\mathcal A}_n(z)}W_h(z),
	\qquad
\widehat\pi_n(z)=\widehat\pi_n(z, {\widehat h}_n(z)).
\end{align}
$\widehat h_n(z)$ is  \emph{optimal} in the sense that it  minimizes the error bound. Then, by construction, we obtain the following error guarantee.
\begin{theorem}
	\label{Th-Lepski}
	Under the stated assumptions, with probability
	at least $1-\alpha$ at any fixed $z$,
	\begin{equation}
		|\widehat\pi_n(z)-\pi(z)|
		\le W_{{\widehat h}_n(z)}(z)
		=\min_{h\in\widehat{\mathcal A}_n(z)}W_h(z).
		\label{eq:compact-guarantee}
	\end{equation}
\end{theorem}

\section{Applications: Uniform confidence bands, testing and change detection}

Let us discuss some  applications of the asymptotic results. Uniform confidence bands go beyond pointwise confidence intervals and are required for inference on the functional form of an estimator.

\begin{example} (Uniform confidence bands) \\
	For this purpose, we  use the standardized version
	\[
	\tilde{c}_n(z,h) = \frac{ \sqrt{n}(\hat c_n(z,h) -c(z,h) ) }{ \sqrt{\hat \zeta_n(z,z,h)} }
	\]
	which converges weakly to the Gaussian mean zero process $  \tilde{ \mathscr{C}}(z,h) = \mathscr{C}(z,h) / \sqrt{\zeta(z,h)} $ with unit variance function, under the assumptions of Theorem~\ref{DonskerCLT}.
	One can simulate trajectories of $ \tilde{\mathscr{C}}(z,h) $ and its supremum $ = \sup_{z \in K} | \tilde{\mathscr{C}}(z,h)| $ to obtain a simulated $(1-\alpha) $-quantile $ q_{c,h}(1-\alpha) $  for some  $ \alpha \in (0,1) $. Then a $ (1-\alpha) $-confidence band for $ c(z,h) $ is given by
	\[
	\left[  \hat{c}_n(z,h) - q_{c,h}(1-\alpha) \sqrt{ \frac{\hat \zeta_n(z,z,h)}{ n} }, \hat{c}_n(z,h)  + q_{c,h}(1-\alpha)  \sqrt{ \frac{\hat \zeta_n(z,z,h)}{ n} } \right].
	\] 
	Analogously, a simulated quantile $ q_{\pi,h}(1-\alpha) $ of the the law of $ \sup_{z \in K} | \mathscr{P}(z,h) | / \sqrt{\eta(z,zh)} $ yields a $ (1-\alpha )$ confidence band 
	\[
	\left[  \hat{\pi}_n(z,h) - q_{\pi,h}(1-\alpha) \sqrt{ \frac{\hat \eta_n(z,z,h)}{ n} }, \hat{\pi}_n(z,h)  + q_{\pi,h}(1-\alpha) 
	\sqrt{ \frac{\hat \eta_n(z,z,h)}{ n } } 
	\right].
	\]
	for $ \pi(z,h) $. 
\end{example}

In many applications, there is some reference  or baseline threshold (at least the constant one), so that one might wish to conduct a statistical test whether or not the reference threshold can be assumed.

\begin{example} (Testing a reference threshold function) \\
	To test the null hypothesis 
	\[
	H_0: c(\cdot) = c_0(\cdot) \qquad \text{versus} \qquad H_1: c(\cdot ) \not= c_0( \cdot ) 
	\]	
	for a given reference threshold function $ c_0: K \to [0, \infty)  $, one may use a Kolmogorov-Smirnov type test statistics such as
	\[
	K_n = \sup_{z \in K} \sqrt{n} \frac{| \hat{c}_n(z,h) - c_0(z,h) |}{\sqrt{\hat \zeta_n(z,z,h)}},
	\]
	which converges in law to $  \tilde{\calM}_h  $ under $ H_0$. Thus, the associated asymptotic test rejects $ H_0 $, if
	$ K_n >  q_{\tilde{\calM},h}(1-\alpha)$.  
\end{example}

The simultaneous confidence band can also be inverted to test a
constant-threshold hypothesis. Define
\[
L_n(z)
=
\widehat c_n(z,h)
-
q_{c,h}(1-\alpha)
\sqrt{\frac{\widehat\zeta_n(z,z,h)}{n}},
\qquad
U_n(z)
=
\widehat c_n(z,h)
+
q_{c,h}(1-\alpha)
\sqrt{\frac{\widehat\zeta_n(z,z,h)}{n}}.
\]
For a specified constant $c_0$, the null hypothesis
\[
H_0:c_h(z)=c_0\quad\text{for all }z\in K
\]
is rejected if the horizontal line $z\mapsto c_0$ is not contained
in the band, that is, if
\[
\sup_{z\in K}
\frac{\sqrt n\,|\widehat c_n(z,h)-c_0|}
{\sqrt{\widehat\zeta_n(z,z,h)}}
>
q_{c,h}(1-\alpha).
\]

For the composite null hypothesis
\[
H_0:\text{there exists }\theta\in\mathbb R
\text{ such that }c_h(z)=\theta
\text{ for every }z\in K,
\]
the confidence-band inversion rejects if no horizontal line is
contained in the band. Equivalently, it rejects if
\[
\sup_{z\in K}L_n(z)>\inf_{z\in K}U_n(z).
\]
This procedure has asymptotic size at most $\alpha$ and may be
conservative.

It is also interesting to examine whether one should instead use two different constant thresholds corresponding to smaller resp. larger values of $ Z $ as a simple and interpretable approach. This corresponds to a change-point test.

\begin{example} (Change-point testing) \\
	Suppose that either the distribution $ \Psi $ of $ U_t = (X_t-\mu)/\sigma )$ or the probability $ \pi(z,h) = P( |Z_t-z| \le h ) $ may be affected by a change, so that the threshold function may depend on time $t$. A suitable change-point model is to assume that for some $ 1 \le t^* $
	\[
	c_t(z,h) =  c_0(z,h) \eins_{\{ t_0 \le t \le t^* \}} + c_1(z,h) \eins_{\{t^* < t \}}, 
	\]
	for two functions $ c_0(\cdot, h ) \not= c_1(\cdot,h) $, the pre- and after-change tresholds. If $ t^* <  n $, $t^* $ is the change-point.
	To test the no-change null hypothesis $ H_0 : t^* = \infty $ against the alternative hypothesis $ H_1 : t^* < n $, one can use the maximized split-sample weighted distance
	\[
	T_n = \max_{\lfloor n t_0 \rfloor \le k < n(1-t_0)} \sup_{z \in K} \frac{k}{n} \frac{n-k}{n} \sqrt{n} | \hat c_{k}(z,h) - \hat{c}_{k+1, n} |.
	\]
	$ T_n $ splits the sample at each candidate change-point location $ k \in \{ \lfloor n t_0 \rfloor, \ldots, \lfloor n (1-t_0) \rfloor \} $, compares the   weighted distance between the estimated threshold calculated from the first $k$ observations and the estimator $ \hat{c}_{k+1, n} $ calculated from the remaining $n-k $ observations, and maximizes this quantity over the candidate locations $k$ and $ z \in K $.  An application of \cite[Prop.~4.1]{DehlingEtAl2014} shows that under $ H_0 $ 
	\[
	T_n \stackrel{d}{\to} \sup_{t \in [t_0,1]} \sup_z | \mathscr{K}(t,z,h) - t \mathscr{K}(1,z,h)|, 
	\]
	as $ n \to \infty $, where  $ \mathscr{K}(t,z,h) $ is the Kiefer process associated to $ t \mathscr{C}(t,z,h) $. 
\end{example}

\section{Real data example and simulation experiments}

We use the FICOS credit risk data set to illustrate the proposed method and compare it with known results. We augment the comparison with results from a FT-transformer network with multi-head attentions and sigmoid output, \cite{Gorishniy2021Revisiting}, in order to check whether they perform better than fully-connected feed forward nets examined in \cite{chen2018interpretable}. 

The FICOS data set of home equity line of credit (HELOC)  was used in the 2018 explainable machine learning competition. It consists of interpretable features and a binary target variable indicating whether or not a borrower was 90 days past due  or worse at least once within the first two years of the credit, see \cite{RudinShap2023}.  After deleting missing values, the data set consists of $ n = 9860 $ observations.  The sample is balanced with $ 52\%$ observations corresponding to credit defaults. In our analysis, the Fico score was used as variable $ X $. As discussed in Section~\ref{sec:representative-fairness}, to ensure algorithmic fairness in the sense of a measure different from representative fairness, one should use a curated data set without such biases. However, since this analysis serves illustrative purposes, we take the data as is.

The loan amount is certainly an important variable, and a system deciding on a credit could be considered fair if for a given loan amount, $L$, the decision against a credit is made with a probability close to the true default frequency in a real training sample where the true default indicator is given. By construction, the approach studied in this paper follows this concept of fairness. Specifically, the question arises whether the classification results can be improved by using a threshold function depending on  the loan amount instead of a fixed threshold, e.g., according to the established Fico scheme that regards a score larger than $670$ as {\em good} and scores exceeding $740$ as {\em very good}. The optimal constant threshold turns out to be $ 712 $ and leads to a classification accuracy of $ ACC = 70.9\% $.

Since the loan amount is highly skewed and the window size $w = 2h$ should be proportional to the loan amount, we set $ Z = \log(L) $ (ranging from $ 9.22$ to $14.97$). The bandwith for estimation was selected manually as well as based on the Lepski method. The manual selection uses $ h = 0.2 $ leading to window sizes which are ca. $40\%$ of the loan amount.  Due to sparsity in the tails, the bandwidth was set to $ 0.5 $ for $z$ less than the $10\% $ or above the $ 90\% $ quantile. Ignoring the tails and focusing on the central part, the corresponding estimated default probabilities have roughly a Lipschitz constant of $ \Delta p / \Delta z \approx L = 0.2 $, which was used for the Lepski method.

\textbf{Fixed bandwidth:} Let us first discuss the results for the manually selected  $h  $. Although the resulting rule is very simple and easy to understand as it compares the customers Fico score to a threshold depending on the loan amount, the resulting in-sample accuracy when trained on the whole data is suprisingly high with $ ACC = 72.3 \% $. When using  $75\% $ of the data for estimation and  the remaining $ 25\%$ as a test set for evaluation, the estimated  accuracy in the test set (averaged over $ 10$ runs) was $ 72.2\%$.

Figure~\ref{Fig1} depicts the estimated default probabilities as a function of the loan amount with $95\%$ confidence bands. Figure~\ref{Fig2} shows the estimated threshold function $ \hat{c}_n(z) $ and the commonly used classification scheme for a FICO score between $590$  (smaller values are classified as poor) and $ 800 $ (larger values are excellent). The uniform confidence band is based on simulated trajectories of the limiting process $ \mathscr{C} $ and somewhat wider than the point-wise confidence intervals. It is clearly seen that no constant line is contained in the band, so that non-constancy of the threshold function is significant on the $5\%$ level. The adapted threshold approach, which ensures the conditional default rates from the training sample, requires a better score for small loans compared to large loans. 

\textbf{Lepski's local bandwidth:} When using the Lepski local bandwidth selection method, accuracy ($72,28\%$), true positive rate ($73.97\%$) and true negative rate ($70.45\%$) are almost the same. But  the fit and confidence bands are much smother in the tails by virtue of its local adaptation to sparse data, see Figures~\ref{Fig3} and \ref{Fig4}. Using more irregular adapted thresholds for small and large loan amounts does not really improve the classification results. 

\begin{figure}
	\begin{center}
		\includegraphics[width=10cm]{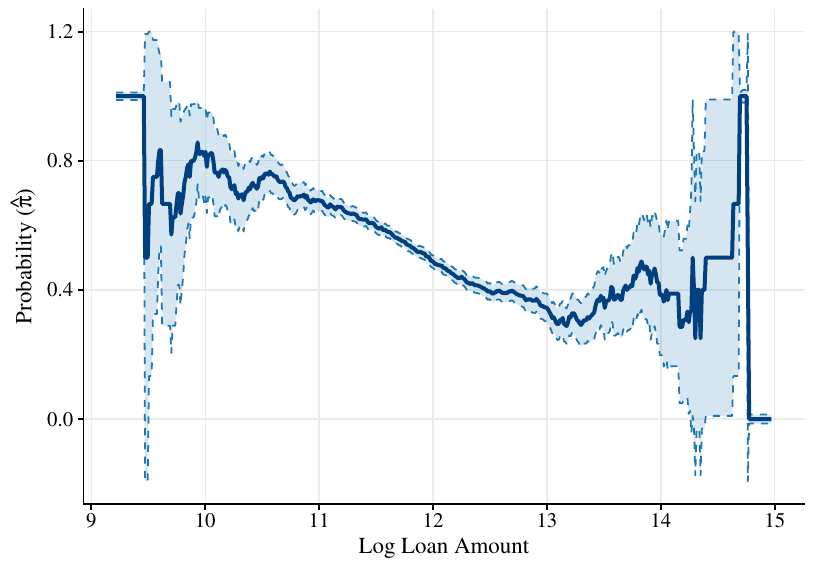}
		\caption{Estimated default probabilities $ \hat{\pi}_n(z) $ given loan amount (bold line) with confidence bands.}
		\label{Fig1} 
	\end{center}
\end{figure}

\begin{figure}
	\begin{center}
		\includegraphics[width=10cm]{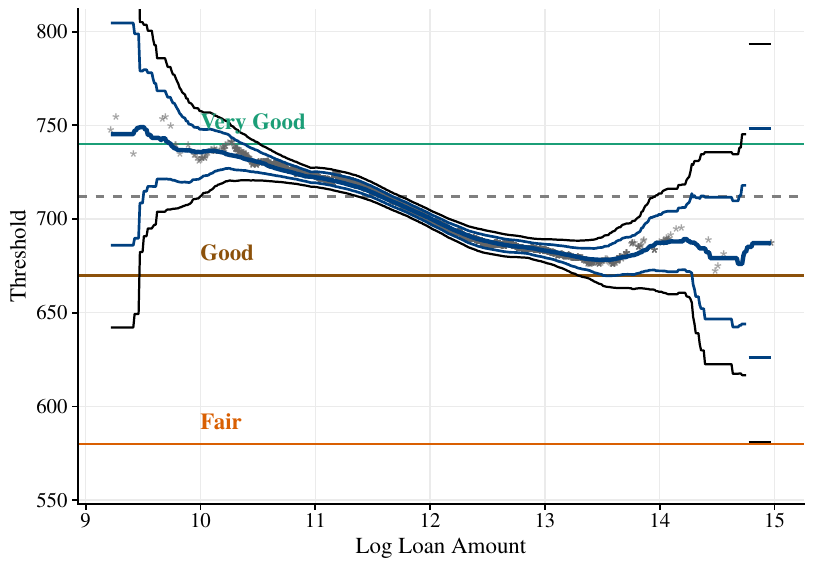}
		\caption{Estimated adapted threshold for the fico score (bold, blue) and optimal constant threshold (dashed), with pointwise confidence intervals (blue) and uniform confidence band (black).}
		\label{Fig2} 
	\end{center}
\end{figure}

\begin{figure}
	\begin{center}
		\includegraphics[width=10cm]{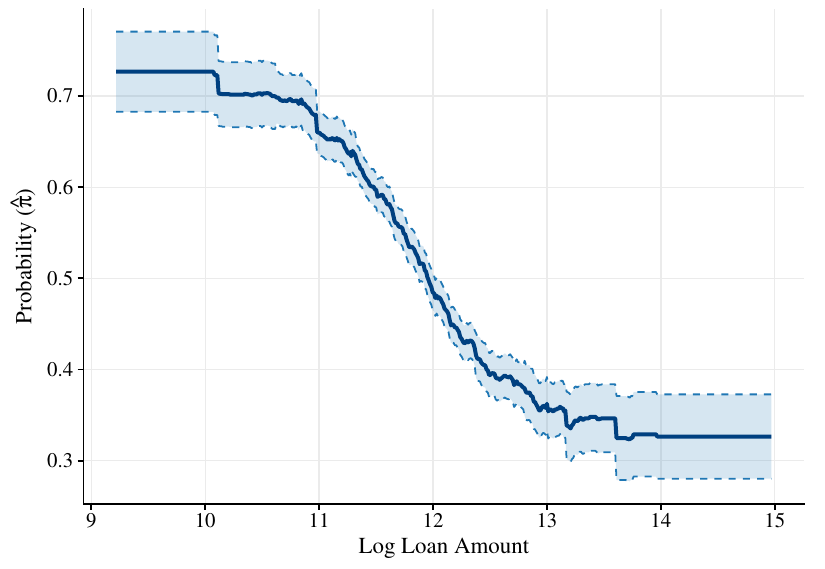}
		\caption{Lepski bandwidth selection: Estimated default probabilities $ \hat{\pi}_n(z) $ given loan amount (bold line) with confidence bands.}
		\label{Fig3} 
	\end{center}
\end{figure}

\begin{figure}
	\begin{center}
		\includegraphics[width=10cm]{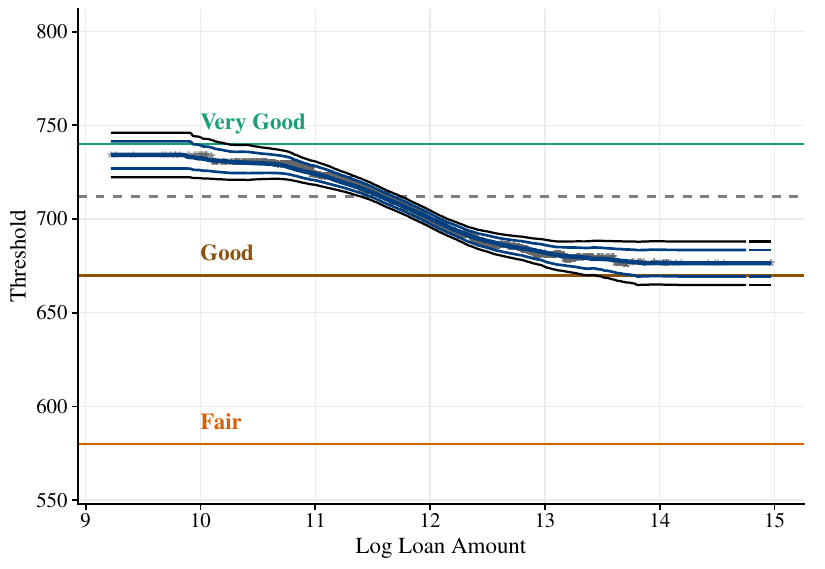}
		\caption{Lepski bandwidth selection: Estimated adapted threshold for the fico score (bold, blue) and optimal constant threshold (dashed), with pointwise confidence intervals (blue) and uniform confidence band (black).}
		\label{Fig4} 
	\end{center}
\end{figure}

\textbf{Comparisons:} 
That accuracy is close to the best methods known from the literature. For example, improves upon a fully connected neural network with eight layers, see \cite{chen2018interpretable}. The dataset was also recently analysed by \cite{RudinShap2023} using several standard machine learning methods (KNN, CART, SVM (lin., RBF, Poly), RF, AdaBoost, Log. Reg.). Averaged over the preprocessing methods used there and trained on $75\% $ of the data, the accurary ranges from $70.7\% $ (KNN) to $72.6\%$  (AdaBoost and Log. Reg.) and is judged as quite similar across methods. Without preprocessing training a logistic regression on the whole data set yields an accuracy of $71.9\% $ (see their tutorial code).

\textbf{Transformer network:} As a more popular method, a transformer network was also used to predict the credit defaults by the fico score $X$ and the log credit amount $Z$ to compare them with the adapted threshold rule of the form $ X > c(Z) $. We also checked other methods including a boosted tree and ogistic regression, but the results were close to those above and are therefore not reported here for brevity. The transformer network consists of pre-layer-normalized multi-head self-attention and GELU feed-forward blocks followed by a normalized linear logit head and a sigmoid to estimate the conditional default probability given $X, Z $. The selected model has width 64, 2 heads, used a dropout of $ 0.1$, batch size $ 256 $ and a weight decay of $ 1e-5$. The average of an esemble of three randomly initialized such networks was used as output. The network is visualized in Appendix~\ref{AppNet}. A cross-validated threshold (yielding $0.5545$ instead of $0.5$) gives $ACC = 71.5\% $, which does not improve upon the interpretable rule obtained the threshold adaptation. 

\textbf{Summary:} These comparisons show that our easily interpretable proposal is on par with such classifiers and even outperforms some of them including a neural network, KNN, CART and SVM (lin.). Furthermore, by the simple form of the decision rule and its application to interpretable features, it is not necessary to generate further summary explanations how it decides. It is also interesting to go beyond accuary.  For the adapted-threshold procedure the true positive rate is $ TPR=74.0\% $, whereas the true negative rate is $ TNR=70.4 \% $. The transformer network yields nicely balanced values, $ TPR = 71.05\% $ and $ TNR = 71.99\%$.

\subsection{Simulation}

To check the accuracy of the proposed inferential procedures, a small simulation study was conducted. The design of the distributional model was guided by the above data example. Samples $ (X_1, Z_1), \ldots, (X_n, Z_n) $ with independent coordinates were simulated according to 
\[
X_i \sim \calN( 11.8, 0.6^2 ), \qquad Z_i \sim \calN( 712, 54^2 )
\]
A true threshold function $ c_0(z) = 800 - 25 (z-9) $ was assumed and the threshold indicator set to 
\[ 
Y_i = \eins_{\{ X_i > c_0(Z_i )\}},
\]
for $ 1 \le i \le n $. Notice that in terms of the data example $ Y_i = 0 $ represents a credit default. For sample sizes $ n \in \{ 100, 250, 500, 1000 \} $ the adapted threshold procedure was applied with parameters as used in the data example and a confidence level of $ 1-\alpha = 90\% $.  Based on $1,000 $ Monte Carlo simulation runs, for each case the coverage probability of the confidence band, the accuracy of classification result based on the rule $ \hat Y_i = \eins_{\{ X_i > \hat c_n(Z_i) \}} $, and the true positive rate and true negative rate was estimated. The results are provided in Table~\ref{tab:simulation_results}.  

\begin{table}[ht]
	\centering
	\caption{Simulation results for various performance metrics across different sample sizes ($n$).}
	\label{tab:simulation_results}
	\begin{tabular}{rcccc}
		\textbf{Sample size $n$} & \textbf{Coverage (\%)} & \textbf{Accuracy (\%)} & \textbf{TPR (\%)} & \textbf{TNR (\%)} \\ \hline
		100  & 97.6 & 94.40 & 95.01 & 92.91 \\
		250  & 90.4 & 95.78 & 96.43 & 94.55 \\
		500  & 88.1 & 96.76 & 97.45 & 95.50 \\
		1,000 & 88.4 & 95.93 & 96.95 & 94.18 \\ \hline
	\end{tabular}
\end{table}

One can notice the good coverage properties and the convincing accuracy which exceeds $ 95\% $ for a sample size of ca. $ 500 $ observations.

\section{Proofs}

\subsection{Nonasymptotic Bounds}

We make use of the following auxiliary result proved in an appendix.

\begin{lemma} 
\label{Lemma_Conditional_Measure}
	Let $ X_1, \ldots, X_n $ be i.i.d. random variables (or vectors) and let $ \hat \mu_n(A) = n^{-1} \sum_{i=1}^n \eins_{\{X_i \in A \}} $ for measurable $ A \subset \R $ be the associated empirical measure. 
	Define $ \hat \mu_n(A|E) = \hat \mu_n(A \cap E) / \hat \mu_n(E) $. 
	Denote 	$ N = \sum_{i=1}^n \one( X_i \in E ) $. Then given $ N=k $, $ k  \in \{  0, \ldots,  n \} $, the random variable
	$  k \hat \mu_n(A|E)  $ follows a binomial law with parameters $ k $ and $ P( X  \in  A |  X \in E ) $, i.e.
	\[
\resizebox{\linewidth}{!}{$\displaystyle P( N \hat \mu_n(A|E)  =  l | N=k) = { k \choose l } P( X  \in  A |  X \in E )^l [1-P( X  \in  A |  X \in E )]^{k-l}, \quad l =0, \ldots, k.$}
\]
	In other words, given $ N  $ the estimator  $ \hat \mu_n(A|E) $ is an empirical measure of $ N $ i.i.d. random variables distributed according to the conditional law $ P(X\in A|X \in E) $. 
\end{lemma}

Fix $ z \in \calZ $ and let $ N = n B_n(z) $ be the random number of points located in $ [z-h, z+h] $. Our analysis is conditional on the event $ N>0 $, such that
\[
  P(N=k|N>0) = \frac{P(N=k,N>0)}{1-P(N=0)} = \frac{{n \choose k} p_B^k(z,h)[1-p_B(z,h)]^{n-k} \one( k > 0)}{1-[1-p_B(z,h)]^n},
\]
for $ k \in \{0, 1, \ldots, n \} $.

\begin{lemma}
\label{LemmaE}
	Let $ p \in (0,1) $ and $ q = 1-p $. 
	\begin{itemize}
	\item[(i)] The inverse of the function
	\[
	  E(t,p) = \frac{1}{1-(1-p)^n} \sum_{k=1}^n {n \choose k} p^k (1-p)^{n-k} \exp(-kt^2), \qquad t > 0,
	\]
	is given by
	\[
	  E^{-1}(u,p) = \sqrt{ - \log\left(  \frac{1}{p} \left[  \{ u(1-(1-p)^n) + (1-p)^n \}^{1/n} - (1-p) \right] \right) }
	\]
	for $ u \in (0,1) $. 
	\item[(ii)] If $ t = E^{-1}(u,p) $ for $ 0 < u \to 0 $ with $ u = o\left( \frac{q^{n}}{1-q^n} \right)$, then 
	\[	
	  2 \sqrt{p} t  + t^2 = 2 \sqrt{p} \sqrt{ \log(1/u)  + \log \frac{n p q^{n-1}}{1-q^n}   }  + \log(1/u) +  \log \frac{n p q^{n-1}}{1-q^n}+ O( u C_n/q^{n-1} ).
	 \]	
	 \end{itemize}
\end{lemma}

\color{black}

To prepare the proof of Theorem~\ref{NonAsympPi}, recall the following concentration inequalities: Let $ Y \sim Bin(n,p) $. By Hoeffding's inequality, for any $ t> 0 $ and all $ n \ge 1 $
\[
	P\left( Y/n - p \ge t \right) \le \exp(-2nt^2), \quad  P\left( Y/n - p \le -t \right) \le \exp(-2nt^2),
\]
yielding the two-sided bound $ P( |Y/n - p| \ge t ) \le 2 \exp(-2nt^2 ) $.
Okamoto, see \cite{okamoto1959some}, shows that for any $ t > 0 $ and all $n \ge 1 $
\begin{align*}
	P\left( \sqrt{Y/n} - \sqrt{p} \ge t \right) & < \exp(-2nt^2) < \exp(-nt^2 )), \\
	P\left( \sqrt{Y/n} - \sqrt{p} \le -t \right) & < \exp(-nt^2).
\end{align*} 	
Hence, since $ (\sqrt{p} - t )^2 - p = -2 \sqrt{p}t + t^2 $,
\begin{align*}
	P\left( Y/n - p \ge 2 \sqrt{p} t + t^2 \right) & \le \exp(-nt^2 ), \\
	P\left( Y /n- p \le - 2 \sqrt{p} t + t^2 \right) & \le \exp(-nt^2),
\end{align*} 	
where $ - 2 \sqrt{p} t + t^2 = -( \sqrt{p}t - t^2) < 0 $ iff $ t < \sqrt{p} $. 
Noting that $ -(2 \sqrt{p} t + t^2) \le  - 2 \sqrt{p} t + t^2 $, we obtain the symmetric bound
\[
	\label{MyIneq} 
	P(|Y/n-p| \ge 2 \sqrt{p} t + t^2 ) \le 2 \exp( - n t^2 ).
\]
Inserting $  \sqrt{2} \bar t  $  for $t $ yields a Hoeffing-type bound,
\begin{equation}
	\label{MyIneq} 
	P(|Y/n-p| \ge 2 \sqrt{2 p} t + t^2 ) \le 2 \exp( - 2 n t^2 ).
\end{equation}
Thus, this inequality improves upon Hoeffding's inequality, if  $ 2( \sqrt{p}  t +  t^2 ) < t \Leftrightarrow t < 1/2 - \sqrt{2p} $. Since $ p < 1/8 $ by assumption, (\ref{MyIneq}) is sharper for $ t < 1/4 $.

\begin{proof}[Proof of Theorem~\ref{NonAsympPi}]

	
	\color{black}
	Notice that $ \hat\pi_n(z,h)  = \hat \mu_n( A | E ) $, if $ \hat \mu_n $ is the empirical measure of $ (Y_i, Z_i) $, $ 1 \le i \le n $, with $ A = \{ (y',z') : y' = 1 \} $ and $ E = \{ (y',z') :  |Z'-z| \le h \} $. Thus, 
	by Lemma~\ref{Lemma_Conditional_Measure}, given $N(z) = n B_n(z) $, $ \hat\pi_n(z,h)  $ is an empirical measure of $N(z) $ Bernoulli variables with success probability $ P( Y= 1 \mid |Z-z| \le h ) = \pi(z,h) $. In view of  \eqref{MyIneq} we have
	\begin{align*}
	  & P( | \hat\pi_n(z,h)  -  \pi(z,h)  | > 2 \sqrt{\pi(z,h)} t + t^2 ) \\
	  & \qquad = E \left[  P( | \hat\pi_n(z,h)  -  \pi(z,h)  | > 2 \sqrt{\pi(z,h)} t + t^2  \mid N(z) ) \right] \\
	  & \qquad \le 2 E\left[  \exp( - N t^2 ) \right].
	\end{align*}
	Using the fact that for any binomial r.v. $ B \sim B(n,p)$, $p \in [0,1]$, one has the bound
	\begin{equation}
	\label{BinomBound}
	  E \exp(-\theta B ) = (1-p + pe^{-\theta})^n \le \exp(- n p(1-e^{-\theta})), \qquad \theta > 0,
	\end{equation}
	we obtain 
	\[
		P( | \hat\pi_n(z,h)  -  \pi(z,h)  | > 2 \sqrt{\pi(z,h)} t + t^2 ) \le 2 \exp(- n \pi(z,h) (1-e^{-t^2}))
	\]
	Therefore, if
	\[
	  t = \sqrt{- \log\left( 1 - \frac{\log( 2/\delta) }{n \pi(z,h)}  \right)}
	\]
	then with probability at least $ 1-\delta $,
	\begin{align*}
	  | \hat\pi_n(z,h)  -  \pi(z,h)  | & \le  2 \sqrt{\pi(z,h)} t + t^2 \\
	   & = 2 \sqrt{\pi(z,h)} \sqrt{- \log\left( 1 - \frac{\log( 2/\delta) }{n \pi(z,h)}  \right)} - \log\left( 1 - \frac{\log( 2/\delta) }{n \pi(z,h)}  \right).
	\end{align*}
	In order to obtain a more accurate bound, we condition on $ N > 0 $ and, instead of using nequality \eqref{BinomBound}, we calculate the expectation $ E\left[  \exp( - N t^2 )  | N > 0 \right] $ exactly. This gives
	\begin{align*}
	  & P( | \hat\pi_n(z,h)  -  \pi(z,h)  | > 2 \sqrt{\pi(z,h)} t + t^2 | N > 0 ) \\
	  & \qquad = \sum_{k=1}^n P( | \hat\pi_n(z,h)  -  \pi(z,h)  | > 2 \sqrt{\pi(z,h)} t + t^2 | N =k ) P( N= k | N > 0) \\
	  & \qquad \le  \frac{2}{1-(1-p_B(z,h))^n} \sum_{k=1}^n {n \choose k} p_B^k(z,h) q_N^{n-k}(z,h) \exp(-k t^2) \\
	  & \qquad = 2 E( t, p_B(z,h) ).
	\end{align*}
	Consequently, by Lemma~\ref{LemmaE}, if for $  0 < \delta < 1 $
	\[
\resizebox{\linewidth}{!}{$\displaystyle t = \sqrt{ - \log\left(  \frac{1}{p_B(z,h)} \left[  \left\{ \frac{\delta}{2} (1-(1-p_B(z,h))^n) + (1-p_B(z,h))^n \right\}^{1/n} - (1-p_B(z,h)) \right] \right) }$}
\]
	then
	\[
		P( | \hat\pi_n(z,h)  -  \pi(z,h)  | \le 2 \sqrt{\pi(z,h)} t + t^2 ) \ge 1-\delta.
	\]
	Lastly, conditioning on $N$ and simply using Hoeffding's bound, we have
	\[
	  P( | \hat\pi_n(z,h)  -  \pi(z,h)  | > t )  \le 2 E\left[  \exp( - 2 N t^2 ) \right]  = 2 E( \sqrt{2} t, p_B(z,h) )
	\]
	Thus, equating $ 2  E( \sqrt{2} t, p_B(z,h) ) = \delta $ gives
	\[
\resizebox{\linewidth}{!}{$\displaystyle t = \frac{1}{\sqrt{2}} \sqrt{ - \log\left(  \frac{1}{p_B(z,h)} \left[  \left\{ \frac{\delta}{2} (1-(1-p_B(z,h))^n) + (1-p_B(z,h))^n \right\}^{1/n} - (1-p_B(z,h)) \right] \right) }$}
\]
	\color{black}
\end{proof}

\subsection{(Functional) central limit theorems}

\begin{proof}[Proof~of~Theorem~\ref{CLTpi}]
	Since the statements are for fixed $z $ and $h $, we write $\hat A_n = \hat A_n(z,h) $, $ B_n = B_n(z,h) $, $ p_A = p_A(z,h) $, $ p_B(z,h) $, and indicate the dependence on them only for the quanties arising in the statements. 
	
	(i) Noting that 
	\[	
	\hat\pi_n(z,h)  = f( A_n, B_n ),  \qquad \text{and} \qquad 
	\pi(z,h) = f( p_A, p_B ),
	\]
	with $ f(x,y) = x/y $ for $ (x, y) \in \R \times \R_{\not=0} $, the proof follows by the Cram\'er-Wold technique and the delta method, as our assumption on $z$ ensures that the gradient $\nabla f$ of $f $ does not vanish at $ (p_A,p_B) $. We have for any $ \lambda, \mu \in \R $
	\[
	\lambda \sqrt{n} (A_n(z,h) - p_A) + \mu \sqrt{n} (B_n(z,h) - p_B) = \frac{1}{\sqrt{n}} \sum_{i=1}^n \xi_i(\lambda, \mu)
	\] 
	where 
	\[
	\xi_i(z,h,\lambda,\mu) = \lambda ( \eins( Y_i = 1, |z-Z_i| \le h) - p_A) + \mu ( \eins( |z-Z_i| \le h ) - p_B),
	\] 
	$ 1 \le i \le n $. Using $ E( \eins( Y_i = 1, |z-Z_i| \le h )  \eins( |z-Z_i| \le h )  ) = p_A $, the variance  $ \Var( \xi_1(z,h, \lambda,\mu) ) $ is given by
	\[
	\eta^2(z,h,\lambda, \mu) =  \lambda^2 p_A(1-p_A) + \mu^2 p_B(1-p_B) + 2 \lambda \mu (p_A - p_A p_B).
	\]
	The Cram\'er-Wold technique shows that 
	\[
	\sqrt{n} \left( \begin{array}{c} A_n(z) - p_A \\ B_n(z) - p_B \end{array} \right) 
	\stackrel{d}{\to} N( \vecnull, \matS ),
	\]
	as $ n \to \infty $, where 
	\[ 
	\matS = \left( \begin{array}{cc} p_A(1-p_A) & p_A(1-p_B) \\ p_A(1-p_B) & p_B(1-p_B) \end{array} \right). 
	\]
	Applying the delta method applied to the function $ f $, we obtain 
	\begin{align*}
		\sqrt{n} [\hat\pi_n(z,h) - \pi(z,h) ] 
		&= \nabla f( p_A, p_B ) \sqrt{n} \left( \begin{array}{c} A_n(z) - p_A \\ B_n(z) - p_B \end{array} \right)  + o_P(1) \\
		& = \frac{1}{p_B} \sqrt{n}( A_n(z) - p_A ) - \frac{\pi(z,h)}{p_B} \sqrt{n} (  B_n(z) - p_B ) + o_P(1).
	\end{align*} 
	Hence, we obtain the asymptotic linearization
	\begin{equation}
		\label{AsympLin}
		\sqrt{n} [\hat\pi_n(z,h) - \pi(z,h) ] = \frac{1}{\sqrt{n}} \sum_{i=1}^n \xi_i( p_B ^{-1}, - \pi(z,h)/p_B ) + o_P(1).
	\end{equation} 
	Now assertion (i) follows from the CLT and Slutzky's lemma noting that
	\begin{align*}
		\eta^2(z,h) &= \eta^2( p_B ^{-1}, - \pi(z,h)/p_B ) \\
		& =\frac{p_A(z,h)(1-p_A(z,h))}{p_B^2(z,h)} + \frac{p_A^2(z,h)(1-p_B(z,h))}{p_B^3(z,h)} - 2 \frac{p_A^2(z,h)(1-p_B(z,h))}{p_B^3(z,h)} \\
		& = \frac{p_A(z,h)(1-p_A(z,h))}{p_B^2(z,h)} - \frac{p_A^2(z,h)(p_B(z,h)-1)}{p_B^3} \\
		& = \frac{\pi(z,h) - p_A(z,h) \pi(z,h) - \pi^2(z,h) + p_A(z,h) \pi(z,h)}{p_B(z,h)} \\
		& = \frac{\pi(z,h)(1-\pi(z,h))}{p_B(z,h)}
	\end{align*}
	Further, the consistency of $ \hat\eta_n^2 $ for $ \eta^2 $ follows from the continuous mapping theorem. Lastly, (i) implies (ii). 
\end{proof}

\begin{lemma}[Uniform Asymptotic Linearity]
	\label{LemmaApproxUniform}
\begin{itemize}
	\item[(i)]
	Let \(K\subset\mathbb R\) be compact, \(t_0\in(0,1)\), and assume
	\[
	\inf_{z\in K}p_{B,h}(z)>0.
	\]
 	Then
 	\[
 	\sqrt n  \{\widehat\pi_{\lfloor nt\rfloor}(z,h)-\pi(z,h)\} = 
 	\left(1 + o_P(1) \right)  \frac{1}{\sqrt{n}} 
 	\frac{n}{\lfloor nt \rfloor } \sum_{i=1}^{\lfloor nt \rfloor } \xi_i(z,h)
 	\]
 	where the $ o_P(1) $ is uniform in $ t \in [t_0,1] $ and $ z \in  K$,
	where
	\[
	\xi_i(z,h)
	=\frac{\mathbf1_{\{|Z_i-z|\le h\}}\{Y_i-\pi(z,h)\}}
	{p_{B,h}(z)}.
	\]
	\item[(ii)] Let $ K = [a,b]$,  \(0<a<b<1\), and assume $ 	\inf_{z\in K}p_{B,h}(z)>0 $ and 
	\[
	0 < a\le\inf_{z\in K}\pi(z,h)
	\le\sup_{z\in K}\pi(z,h)\le b < 1
	\]
	for constants $ a, b $. 	Let \(q(r)=\Psi^{-1}(1-r) \) be twice continuously differentiable on
	\([a,b]\). Then
	\[
		\sqrt n\{\widehat c_{\lfloor nt \rfloor}(z,h)-c_h(z)\} = (1+o_P(1)) \frac1{\sqrt n} \frac{n }{\lfloor nt \rfloor} \sum_{i=1}^{\lfloor nt \rfloor}
		q'(\pi(z,h))\xi_i(z,h),
	\]
	where the $ o_P(1) $ is uniform in $ t \in [t_0,1] $ and $ z \in K $.
\end{itemize}
\end{lemma}

\begin{proof}
	Clearly, by Example 2.5.4 and Example 2.10.10 of \cite{VaartWellner2023}, the class of functions
$ \calG = \{  g_z : z  \in K \} $  inducing the processes $ B_{\lfloor mt\rfloor}(z,h) $, given by the functions
\[  
g_z(z',y') = \eins_{\{|z'-z| \le h\}}, \qquad (y',z') \in \{0,1\} \times \R, z \in K,
\]
is a Glivenko-Cantelli class. Therefore, by the sequential uniform central limit theorem, 
\[
\sup_{t,z}|B_{\lfloor nt\rfloor}(z,h)-p_B|= o_1(1), \qquad P-a.s.,
\]
such that 
\[
\sup_{t,z}| 1/B_{\lfloor nt\rfloor}(z,h)-1/p_B|= o_1(1), \qquad P-a.s.,
\]
since $ \inf_{t,h} p_B(z,h) > 0 $ by assumption. Now, for any $ k \in \N $ such that $ B_k(z,h) > 0 $ we have the representation
	\begin{align*}
	\widehat\pi_k(z,h)-\pi(z,h)
	&=\frac{(A_k(z,h)-p_A(z,h))-\pi(z,h)(B_k(z,h)-p_B(z,h))}{B_k(z,h)} \\
	& = \frac{1}{B_k(z,h)} \frac{1}{k} \sum_{i=1}^k \eins( |Z_i-z| \le h ) \{ Y_i - \pi(z,h) \},
	\end{align*}
	yielding
	\begin{align*}
		\sqrt n(\widehat\pi_{\lfloor nt \rfloor}(z,h)-\pi(z,h))
		&=  \frac{p_B(z,h)}{B_{\lfloor nt \rfloor}(z,h)} \frac{1}{\sqrt{n}} 
		\frac{n}{\lfloor nt \rfloor } \sum_{i=1}^{\lfloor nt \rfloor } \xi_i(z,h) \\
		& = \left(1 + o_P(1) \right)  \frac{1}{\sqrt{n}} 
		\frac{n}{\lfloor nt \rfloor } \sum_{i=1}^{\lfloor nt \rfloor } \xi_i(z,h),
	\end{align*}
	where the $ o_P(1) $ is uniform in $ t,z $, since 
	\[
	\sup_{t,z} \left| \frac{p_B(z,h)}{B_{\lfloor nt \rfloor}(z,h)}-1 \right|
		  =o_P(1), \qquad n \to \infty.
	\]
 This verifies (i). 	
To see the second assertion, notice that the uniform consistency of \(\widehat\pi_{\lfloor nt \rfloor }\), $t \in [t_0,1] $, and boundedness of \(q''(u) = -\psi'( \Psi^{-1}(1-u))/\psi( \Psi^{-1}(1-u))\) on \([a,b]\)  imply that in the second-order Taylor expansion
	\[
\widehat c_{\lfloor nt \rfloor }(z,h)-c_h(z)
=
q'(\pi(z,h))
\{\widehat\pi_{\lfloor nt \rfloor}(z,h)-\pi(z,h)\}
+R_n(t,z),
\]
the remainder satisfies
\[
\sup_{t,z}|R_n(t,z)|
\le
\frac12\|q''\|_{\infty,[a,b]}
\sup_{t,z}
|\widehat\pi_{\lfloor nt \rfloor}(z,h)-\pi(z,h)|^2
=o_P(n^{-1/2}),
\]
since, as shown below at the beginning of the proof of Theorem~\ref{DonskerCLT}, $  \frac{1}{\sqrt{n}} 
\sum_{i=1}^{\lfloor nt \rfloor } \xi_i(z,h)$ converges for $ t = 1 $ weakly to a tight Gaussian process indexed by $ f \in \mathcal{F} $, and by independence of the summands this implies the weak convergence of the sequential process in $t \in [t_0,1] $ and $ f \in \mathcal{F} $, see the detailed discussion below, such that 
\[
\sup_{t,z} \sqrt{n} |\widehat\pi_{\lfloor nt \rfloor}(z,h)-\pi(z,h)| = O_P(1).
\]
\end{proof}

\begin{proof}[Proof of Theorem~\ref{DonskerCLT}]
	Consider the class $ \calF = \{ f_z : z \in K \} $ of functions $ \R \times \{0,1\} \to \R $, where
	\[
	f_z( x, y ) = \frac{1}{p_B(z,h)} \eins( |x-z| \le h, y = 1) - \frac{\pi(z,h)}{p_B(z,h)} \eins( |x-z| \le h ) ), \qquad x \in \R, y \in \{ 0,1 \}, 
	\]
	for $ z \in \R$, and notice that $ f_z(Z_i,Y_i) = \xi_i(z,h) $.	
	$ \calF  $ is a $ P $-Donsker class, see Example 2.5.4 and Example 2.10.9 of \cite{VaartWellner2023}. Therefore, the empirical process $ \mathscr{P}_n(f) = \sqrt{n}(P_n-P)(f) $, $ f \in \calF $, given for $ f = f_z $, $ z \in K $, by
	\[
	\mathscr{P}_n(f_z) = \sqrt{n}(P_n-P)(f_z) = \frac{1}{\sqrt{n}} \sum_{i=1}^n \xi_i( z, h ),
	\]
	converges weakly in  $ l^\infty( \calF) $ to a mean zero tight Borel measurable Brownian bridge process $ \mathscr{P} \in l^\infty(\calG) $ with covariance function
	\[
	\Cov( \mathscr{P}(f_u), \mathscr{P}(f_{v}) ) = E(f_u f_v) - E(f_u) E(f_v), \qquad u, v \in K,
	\]
	where $ E(f_u) = 0 $ and
	\begin{align*}
		E(f_u f_v) & = \frac{1}{p_B(u,h)p_B(v,h)} P(|Z_1-u| \le h, |Z_1-v| \le h, Y_1=1) \\
		& \quad + \frac{\pi(u,h)\pi(v,h)}{p_B(u,h)p_B(v,h)} P(|Z_1-u|\le h, |Z_1-v|\le h) \\
		& \quad - \frac{\pi(v,h)}{p_B(u,h)p_B(v,h)} P( |Z_1-u|\le h, |Z_1-v|\le h, Y_1=1) \\
		& \quad - \frac{\pi(u,h)}{p_B(u,h)p_B(v,h)} P( |Z_1-u|\le h, |Z_1-v|\le h, Y_1=1) \\    
		& = \frac{1}{p_B(u,h)p_B(v,h)} \biggl( (1-\pi(u,h) - \pi(v,h)) P( |Z_1-u|\le h, |Z_1-v|\le h, Z_1=1)  \\
		& \qquad + \pi(u,h) \pi(v,h) P(|Z_1-u|\le h, |Z_1-v|\le h)  \biggr) \\
		& = \frac{1}{p_B(u,h)p_B(v,h)} \left( (1-\pi(u,h) - \pi(v,h)) \gamma_2(u,v,h)  + \pi(u,h) \pi(v,h) \gamma_1(u,v,h) \right), 
	\end{align*}
	with 
	\begin{align*} \gamma_1(u,v,h) &= P(|Z_1-u|\le h, |Z_1-v|\le h), \\ 
		\gamma_2(u,v,h) &= P( |Z_1-u|\le h, |Z_1-v|\le h, Y_1=1),
	\end{align*}  
	for $ u, v \in K $. Combining $ \calE_n \Rightarrow \calE $, $ n \to \infty $, and Lemma~\ref{LemmaApproxUniform}, which implies
	\[
	\sup_{z \in K} | \sqrt{n} ( \hat \pi_n(z,h) - \pi(z,h) ) - \mathscr{P}_n(f_z) | = o_P(1),
	\]
	an application of Slutzky's lemma yields the weak convergence 
	\[ 
	\{ \sqrt{n} ( \hat \pi_n(z,h) - \pi(z,h) ) : z \in K \} \Rightarrow \{ \mathscr{P}(z,h) : z \in K \},
	\]
	as $ n \to \infty $, in $ l^\infty(\R) $. Now the continuous mapping theorem entails
	\[
	\sup_{z \in K} \sqrt{n} | \hat \pi_n(z,h) - \pi(z,h)  | \stackrel{d}{\to} \sup_{z \in K} | \mathscr{P}(z,h) |,
	\]
	as $ n \to \infty $. The proof for the process $ \mathscr{C}_n(z,h) $ follows the same arguments using the asymptotic linearity again. Taking account of the asymptotic scaling with the factor $ g(z,h) := -1/\psi( \Psi^{-1}(1-\pi(z,h))) $ which is positive for all $ z \in K $ by assumption, let us consider the class $ \tilde \calF = \{ \tilde f_z : z \in K \} $ of functions $ \R \times \{0,1\} \to \R $, where
	\[
	\tilde f_z( x, y ) = \frac{g(z,h)}{p_B(z,h)} \eins( |x-z| \le h, y = 1) - \frac{\pi(z,h)g(z,h)}{p_B(z,h)} \eins( |x-z| \le h ) ), \qquad x \in \R, y \in \{ 0,1 \}, 
	\]
	for $ z \in \R$. Again, this is a $ P $-Donsker class by virtue of Example 2.5.4 and Example 2.10.9 of \cite{VaartWellner2023}, such that $\mathscr{C}_n(z,h) = (P_n-P)(\tilde f_z ) $, $ z \in K $,  converges weakly to a $P$-Brownian bridge $ \mathscr{C}(\tilde f_z ) $, $ z \in K $, as $ n \to \infty $. Combining this with Lemma~\ref{LemmaApproxUniform}, which yields
	\[
	\sup_{z \in K} | \sqrt{n} ( \hat c_n(z,h) - c(z,h) ) - \mathscr{C}_n(f_z) | = o_P(1),
	\]
	as $ n \to \infty $, Slutzky's lemma entails the claimed weak convergence
	\[
	\mathscr{C}_n \Rightarrow \mathscr{C}, \qquad n \to \infty.
	\]
	Lastly, the covariance function $ \zeta(u,v,h) $ follows by a simple calculation.  This completes the proof for the empirical processes $ \mathscr{P}_n $ and $ \mathscr{C}_n $ indexed by $ K $. 
	
	Since the approximating linear processes take the form of empirical processes  indexed by $P$-Donsker function classes $ \calF $ and $ \tilde \calF $, respectively, and are calculated from a sequence of i.i.d. random variables, the classes $ \calF $ and $ \tilde \calF $ are functionally Donsker as well, \cite[Th.~2.12.1]{VaartWellner2023}. By definition, this means that the sequential processes 
	\[
	  \mathscr{Q}_n^{lin}(t,z,h) = \frac{1}{\sqrt{n}} \sum_{i=1}^{\lfloor nt \rfloor} \xi_i(z,h) \qquad \text{and} \qquad \mathscr{D}_n^{lin}(t,z,h) = \frac{1}{\sqrt{n}} \sum_{i=1}^{\lfloor nt \rfloor} \zeta_i(z,h)
	\]
	converge in distribution in $ l^\infty( [0,1] \times K; \R ) $ to  tight mean zero Kiefer-M\"uller processes $ \mathscr{Q}^{lin}(t,z,h) $ and $ \mathscr{D}^{lin}(t,z,h)$ , $ (t,z) \in [0,1] \times K $, with covariance functions
	\[
	  \Cov( \mathscr{Q}^{lin}(s,u,h), \mathscr{Q}^{lin}(t,v,h)) = (s \wedge t) \eta(u,v,h)
	\]
	and
	\[
		\Cov( \mathscr{D}^{lin}(s,u,h), \mathscr{D}^{lin}(t,v,h)) = (s \wedge t) \zeta(u,v,h),
	\]
	respectively. By Lemma~\ref{LemmaApproxUniform}, applied to the processes scaled by the factor $  \frac{ n }{ \lfloor nt \rfloor }$,  we may apply Slutzky's lemma to conclude that, for $ t\in [t_0,1] $ and $ z \in K $,
	$ \mathscr{P}_n(t,z,h) $ converges weakly to the mean zero Gaussian processes $ \mathscr{P}(t,z,h ) = t^{-1} \mathscr{Q}^{lin}(t,z,h) $ and $ \mathscr{C}_n(t,z,h)  $ to $ \mathscr{C}(t,z,h) =  t^{-1} \mathscr{D}^{lin}(t,z,h) $, as $ n \to \infty $,
	in $ l^\infty( [t_0,1] \times K; \R ) $. Clearly, the covariance functions are given by
	\[
	  \Cov( \mathscr{P}(s,u,h), \mathscr{P}(t,v,h) ) =  \frac{(s \wedge t)}{st} \eta(u,v,h) 
	\]
	and
	\[
	\Cov( \mathscr{C}(s,u,h), \mathscr{C}(t,v,h) ) =   \frac{(s \wedge t)}{st} \zeta(u,v,h), 
	\]
	for $ s, t \in [t_0,1] $ and $ u,v \in K $.	
\end{proof}

\begin{proof}[Proof of Theorem~\ref{Th-Lepski}]  
	Write $\overline\pi_h=m_h^{-1}\sum_{i\in N_h}\pi(Z_i)$.
	Using Lemma~\ref{Lemma_Conditional_Measure},  
	Hoeffding's inequality and a union bound over the $J$ bandwidths yields
	$|\widehat\pi_h-\overline\pi_h|\le s_h$ simultaneously with probability
	at least $1-\alpha$. When using Lemma~\ref{Lemma_Conditional_Measure} and Okamoto's inequality,
	$\overline\pi_h-\widehat\pi_h
	\le2\sqrt{\overline\pi_h}t_h+t_h^2$ implies
	$\sqrt{\overline\pi_h}\le t_h+\sqrt{\widehat\pi_h+2t_h^2}$.
	Hence $\overline\pi_h\le U_h$, giving the same simultaneous bound
	with \eqref{eq:compact-okamoto}.
	The Lipschitz assumption gives
	$|\overline\pi_h-\pi(z)|\le B_h$, so
	$|\widehat\pi_h-\pi(z)|\le W_h$ for all $ h \in \mathcal{H} $.
	The triangle inequality then gives
	$|\widehat\pi_h-\widehat\pi_g|\le W_h+W_g$ for all $ g, h \in \mathcal{H} $ with $ g < h $. This proves \eqref{eq:compact-guarantee}.
\end{proof}

\bibliography{lit}

\appendix

\section{Degeneracy of the weak limit}

The degeneracy of the weak limit of the Parzen-Rosenblatt density process under the asymptotic regime $ nh \to \infty $ and $h \to 0 $, although the finite-dimensional distributions converge to non-degenerate Gaussian laws, carries over to Nadaraya-Watson type estimators such as $ \hat{m}_n(x,h) $ as defined in Theorem~\ref{IP-impossible}. This holds true, although certain functionals such as integral means of such processes may converge to non-degenerate limits. The basic reason is that integrals, generalizing the linear combinations in finite dimensional spaces to infinite-dimensional ones, have a different convergence rate. 

\begin{proof}[Proof of Theorem~\ref{IP-impossible}] The proof follows as in \cite{Nishiyama01032011}. For completeness, we provide details. By \cite{VaartWellner2023} it suffices to show that $ < G, g >  = \int G(x) g(x) \, dx \stackrel{d}{=} 0 $ for all $ g \in L_2(\R)$. By continuity of the inner product, this follows from 	
	 $ \int M_n(x,h) g(x) \, d x \stackrel{P}{\to} 0 $ (equivalently $ \stackrel{d}{\to} 0 $), as  $ h \to 0 $ and $ nh \to \infty $, for any $ g \in L_2(\R )$. Let $ H = \{ e_j : j \ge 1 \} $ be the weighted Hermite polynomials which span $ L_2(\R) $ with respect to the inner product $ < \cdot, \cdot > $ and are bounded. Denote by $C_H $ a norm constant of $H$.  We may write $ h = \sum_{i=1}^\infty e_i < h, e_i> $. Since 
	 \[ 
	 \| \int G(x) \sum_{i > k} e_i(x)  < h, e_i > \, dx \|_{L_1} \le \| G \|_{L_2} \sum_{i>k, j>k} <e_i,e_j>  <h,e_i><h,e_j> \to 0,
	 \] 
	 for $k \to \infty$, so that $ \int G(x) \sum_{i > k} e_i(x)  < h, e_i> \, dx  $ converges to $0$ in probability, as $ k \to \infty $, it suffices to show that  $  \| < M_n(x,h),  e_j >  \|_{L_2} \to 0  $, as $ nh \to \infty $ and $h \to 0 $, since this implies
	\[
	\sum_{j=1}^k < M_n(x,h),  e_j >  \stackrel{L_2}{\to} 0,
	\]
	as $ nh \to \infty $ and $h \to 0 $,  Let $ r_n = \sqrt{nh} $. We have
	\[
	< M_n(x,h), e_j >  = \sum_{i=1}^n [V_{ni} - E(V_{ni}) ]
	\]
	where $ V_{ni} = \frac{\sqrt{nh}}{n} \int \frac{1}{h} \eins_{\{ |X_i-x|\le h \}} Y_i e_j(x) \, dx $, $ 1 \le i \le n$. Noting that
	\[
	\left[ \int \frac{1}{h} \eins_{\{ {u-x|\le h}\}} y e_j(x) \, dx \right]^2 \le 2 C_H y^2
	\]
	the variances of the $ V_{ni} $ can be bounded by 
	\begin{align*}
		E(V_{ni}^2) + (E V_{ni})^2  
		& = \frac{nh}{n^2} \int \left[ \int \frac{1}{h} \eins_{\{ {u-x|\le h}\}} y e_j(x) \, dx  \right]^2 f(u,y) d(u,y)  \\
		& \qquad + \frac{nh}{n^2} 
		\left\{ \int \left[  \int \frac{1}{h} \eins_{\{ |u-x| \le h\}} y e_j(x) \, dx  \right] f(u,y)  \right\}^2 \\
		& \le 4 C_H^2 E(Y_1^2) \frac{h}{n}.
	\end{align*}
	Hence, 
	\[
	\Var(  <M_n(x,h), e_j >  ) = \sum_{i=1}^n \Var(V_{ni}) = O(h),
	\]
	which completes the proof. 
\end{proof}

\section{Additional proofs}

\begin{proof}[Proof of Lemma~\ref{Lemma_Conditional_Measure}]
	Denote $ P_k( \cdot ) = P( \cdot | N=k )$. Let $ S = \{ i \in \{ 1, \ldots, n \} : X_i \in E \} $.
	Notice that $ \{ N= k \} = \{ |S| = k \} $. Clearly, given $N = k $ the random set $S$ is uniformly distributed on the set $ \calS_k $ of subsets of $ \{ 1, \ldots, n \} $ with $k$ elements. 
	Further, for any $ T \in \calS_k $ it holds $  \{ S=T \} = \{ S = T, N=k \} $, which implies
	$ P_k( \cdot | S=T ) = P( \cdot | S=T) $. By conditioning on the possible values of $ S $ we obtain
	\begin{align*}
		P_k( N \hat \mu_n(A|E)  = l ) 
		&= \frac{1}{{k \choose l}} \sum_{T \in \calS_k}
		P( k \hat \mu_n(A|E) = l | S = T ) \\
		& = \frac{1}{{k \choose l}} \sum_{T \in \calS_k}
		P\left(  \sum_{i=1}^n \one( X_i \in A, X_i \in E ) = l \mid S = T  \right)  \\
		& = \frac{1}{{k \choose l}} \sum_{T \in \calS_k}
		P\left(  \sum_{i \in T} \one( X_i \in A ) = l \mid S = T  \right) 
	\end{align*}
	Since $ \{ S = T \} = \{ X_i \in E, i \in T, X_i \not\in E, i \not\in T \} $, for $ \vecx = (x_i)_{i\in T} \in \R^k $
	\begin{align*}
		P( (\eins(X_i\in A))_{i \in T} = \vecx | S= T) &= \frac{P(\eins(X_i\in A))_{i \in T} = \vecx, X_i \in E, i \in T) }{P( X_i \in E, i \in T)} \\
		& = \frac{\prod_{i \in T} P( \eins(X_i \in A) = x_i, X_i \in E)} { \prod_{i \in T} P( X_i \in E )} \\
		& = \prod_{i \in T} P( \one(X_i \in A) = x_i | X_i \in E ).
	\end{align*}
	Thus, for fixed $ T \in \calS_k $, $ \eins(X_i \in A) $, $i \in T $, are $|T|=k$ random Bernoulli variables being i.i.d. under $P(\cdot|S=T)$ with success probability $P( X_i \in A | X_i \in E )$. Plugging in the corresponding conditional binomial probabilities under $P(  \cdot | S=T) $ for $ \{ \sum_{i \in T} \one( X_i \in A ) = l  \} $, $ l \in \{ 0, \ldots, k \} $, completes the proof. 
\end{proof}

\begin{proof}[Proof of Lemma~\ref{LemmaE}]
	Using the fact that $ \sum_{k=0}^n {n \choose k} (p e^{-t^2})^k (1-p)^{n-k} = ( p e^{-t^2} + (1-p) )^n $ the first result follows by a straightforward calculation. 
	
	To show the second assertion, let $ C_n = 1-q^n $. Write
	\[
	u = f(t) =  \frac{1}{C_n} \left[ ( p e^{-t^2} + q)^n  - q^n \right]
	\]
	Note that $ u = f(t) $ iff. $ C_n u + q^n = (p e^{-t^2} + q )^n $ iff. $ p e^{-t^2} = (q^n + C_n u)^{1/n} -q $. 
	The Taylor expansion of $ g(x) = (1+x)^a $, $0 < a <  1 $, yields $ (1+x)^a = 1 + a x + a(a-1) x^2/4 + O(x^3)$, with $ |a(a-1)| = (n-1)/n^2 \le 1 $ if $ a = 1/n$, for $ 0 < x < 1 $, since $ |g''(x)| = |a(a-1)(1+x)^{a-2}| \le 1$ for all $|x|<1 $, if $ a=1/n$. Hence, if $ u (1-q^n) / q^n < 1 $, then for some $ 0 < \xi < u $
	\begin{align*}
		p e^{-t^2} &= q (1+ u C_n / q^n )^{1/n}  - q \\
		& = q( 1 + u  C_n / n q^n + (1/4) \xi^2 C_n^2 / n q^{2n}  ) - q \\
		& = u C_n /n q^{n-1} + (1/4) \xi^2 C_n^2 / n q^{2n-1}.
	\end{align*}
	Next recall $ \log(1+x) = x + O(x^2) $, $ |x| < 1 $. Provided $ (1/4) u C_n/q^{n-1} = (1/4) u(1-q^n)/q^n < 1 $ we obtain
	\[
	t^2 = - \log\left( u  \frac{C_n}{n p q^{n-1}}( 1  + (1/4)(\xi^2/u) C_n/q^n )  \right) =  -\log(u)  - \log( C_n/n p q^{n-1}) + O( u C_n/q^n ).
	\]
	Therefore, if $ u (1-q^n) / q^n < 1 $, we can conclude 
	\[
	t = \sqrt{ -\log(u)  - \log( (1-q^n)/n p q^{n-1})  }  + O( u (1-q^n)/q^n )
	\]
	yielding
	\begin{align*}
		2 \sqrt{p} t  + t^2 & = 2 \sqrt{p} \sqrt{ \log(1/u)  + \log( n p q^{n-1}/(1-q^n))  }  + \log(1/u) + \log( n p q^{n-1}/(1-q^n)) \\
		& \qquad  + O( u C_n/q^{n-1} ),
	\end{align*}
	where the $ O $ term is $ o(1) $ if $ u = o( q^{n-1}/(1-q^n) ) $. 
\end{proof}

\section{Transformer network}
\label{AppNet}

\usetikzlibrary{arrows.meta,positioning,calc}
\begin{tikzpicture}[
	font=\small,
	>=Latex,
	flow/.style={->,thick,draw=black!75},
	box/.style={draw=black!65,rounded corners=2pt,align=center,
		minimum height=7mm,text width=5.5cm,inner sep=4pt,fill=black!3},
	block/.style={box,fill=blue!7,draw=blue!55!black},
	operation/.style={box,text width=4.5cm},
	sum/.style={draw=black!70,circle,minimum size=6mm,inner sep=0pt,fill=white},
	note/.style={font=\footnotesize,align=center,text width=5.7cm}
	]
	\node[font=\bfseries] at (3,0) {(a) Network Ensemble $ k =1, 2, 3 $};
	\node[font=\bfseries] at (10.7,0) {(b) Inside each Transformer block};
	
	\node[box] (input) at (3,-0.9)
	{Two numerical inputs\\$x_1=\mathrm{FICO},\quad x_2=\mathrm{logamount}$};
	\node[box] (scale) at (3,-2.0)
	{Training-set standardization\\$\widetilde x_j=(x_j-\mu_j)/s_j$, $j=1,2$};
	\node[block] (token) at (3,-3.2)
	{feature embeddings\\
		$\boldsymbol t_j=\widetilde x_j\boldsymbol w_j+\boldsymbol b_j\in\mathbb R^{64}$};
	\node[block] (sequence) at (3,-4.4)
	{Prepend learned classification token\\
		$H^{(0)}=[\boldsymbol t_{\mathrm{CLS}};\boldsymbol t_1;\boldsymbol t_2]
		\in\mathbb R^{3\times64}$};
	\node[block] (block1) at (3,-5.5)
	{Transformer block 1\quad $(3\times64)$};
	\node[block] (block2) at (3,-6.5)
	{Transformer block 2\quad $(3\times64)$};
	\node[box] (cls) at (3,-7.5)
	{Extract final CLS representation\\
		$\boldsymbol h_{\mathrm{CLS}}^{(2)}\in\mathbb R^{64}$};
	\node[box] (head) at (3,-8.8)
	{LayerNorm $\longrightarrow$ Linear $(64\to1)$\\
		$\ell_k=\boldsymbol a_k^\top
		\operatorname{LN}(\boldsymbol h_{\mathrm{CLS}}^{(2)})+b_k$};
	\node[box] (prob) at (3,-10.0)
	{Sigmoid output\\$p_k(x)=\{1+\exp(-\ell_k)\}^{-1}$};
	\node[block] (average) at (3,-11.5)
	{Three-member probability ensemble\\[2pt]
		$\widehat p(x)=\{p_1(x)+p_2(x)+p_3(x)\}/3$};
	\foreach \a/\b in {input/scale,scale/token,token/sequence,sequence/block1,
		block1/block2,block2/cls,cls/head,head/prob,prob/average}
	\draw[flow] (\a) -- (\b);
	
	\node[operation] (hin) at (10.7,-0.9) {Input $H\in\mathbb R^{3\times64}$};
	\coordinate (branch1) at (10.7,-1.5);
	\node[operation] (norm1) at (10.7,-2.0) {LayerNorm};
	\node[operation,fill=blue!7] (attention) at (10.7,-3.1)
	{Multi-head self-attention\\2 heads; head dimension 32};
	\node[operation] (drop1) at (10.7,-4.1) {Dropout $(0.1)$};
	\node[sum] (add1) at (10.7,-4.95) {$+$};
	\coordinate (branch2) at (10.7,-5.5);
	\node[operation] (norm2) at (10.7,-6.0) {LayerNorm};
	\node[operation,fill=blue!7,font=\footnotesize] (ffn) at (10.7,-7.25)
	{Token-wise feed-forward\\
		Linear $(64\to128)$; GELU\\
		Dropout $(0.1)$; Linear $(128\to64)$};
	\node[operation] (drop2) at (10.7,-8.5) {Dropout $(0.1)$};
	\node[sum] (add2) at (10.7,-9.35) {$+$};
	\node[operation] (hout) at (10.7,-10.3) {Output $H'\in\mathbb R^{3\times64}$};
	\foreach \a/\b in {hin/norm1,norm1/attention,attention/drop1,drop1/add1,
		add1/norm2,norm2/ffn,ffn/drop2,drop2/add2,add2/hout}
	\draw[flow] (\a) -- (\b);
	\draw[flow] (branch1) -- (7.7,-1.5) |- (add1.west);
	\draw[flow] (branch2) -- (7.7,-5.5) |- (add2.west);
	\fill (branch1) circle (1.3pt);
	\fill (branch2) circle (1.3pt);
	\node[rotate=90,font=\footnotesize,fill=white,inner sep=2pt] at (7.7,-3.1)
	{residual connection};
	\node[rotate=90,font=\footnotesize,fill=white,inner sep=2pt] at (7.7,-7.35)
	{residual connection};
	\node[note] at (10.7,-11.55)
	{Same topology; distinct block parameters.\\
		Attention-weight dropout is also $0.1$.\\
		All dropout is disabled at prediction time.};
\end{tikzpicture}

\end{document}